\documentclass[11pt,reqno,a4paper]{amsart}
\usepackage{amssymb,dsfont,mathtools}
\usepackage{pifont}
\usepackage{color}
\usepackage{bm}
\usepackage{graphicx}
\usepackage{subcaption}
\usepackage{diagbox}
\usepackage{cite}
\usepackage{float}
\definecolor{AMSSblue}{cmyk}{0.94,0.75,0.1,0}
\definecolor{AMSSgreen}{RGB}{0, 127, 0}
\definecolor{AMSSyellow}{RGB}{240,140,0}
\definecolor{AMSSred}{RGB}{210, 26, 66}

\usepackage[colorlinks,linkcolor=AMSSred,anchorcolor=AMSSred,citecolor=AMSSgreen]{hyperref}
\newtheorem{theorem}{Theorem}[section]
\newtheorem{lemma}[theorem]{Lemma}
\newtheorem{proposition}[theorem]{Proposition}
\newtheorem{corollary}[theorem]{Corollary}

\theoremstyle{definition}
\newtheorem{definition}[theorem]{Definition}

\theoremstyle{remark}
\newtheorem{remark}{Remark}[section]
\numberwithin{equation}{section}

\newcommand{\dd}{{\rm d}}
\newcommand{\ee}{{\rm e}}
\newcommand{\s}{{\sharp}}
\newcommand{\f}{{\flat}}
\newcommand{\n}{{\natural}}
\newcommand{\LL}{{\mathfrak{L}}}

\newcommand{\DD}{{\mathfrak{D}}}
\newcommand{\qq}{{\mathfrak{q}}}
\allowdisplaybreaks[4]

\begin{document}
	\title[Inverse Piston Problem]{An Inverse Problem for Determining the Piston Speed from a Given Lipschitz Leading Shock}
	
	\author{Gui-Qiang G. Chen}
	\address{Mathematical Institute, University of Oxford,
		Oxford, OX2 6GG, UK; School of Mathematical Sciences, Fudan University, Shanghai 200433, China. }
	\email{\tt  gui-qiang.chen@maths.ox.ac.uk}
	\thanks{Gui-Qiang G. Chen was supported in part by the UK Engineering
		and Physical Sciences Research Council Awards 
		EP/V008854 and EP/V051121/1. }
	\author{Qianfeng Li}
	\address{Department of Mathematics, Friedrich--Alexander--Universit\"at Erlangen--N\"urnberg, Cauerstr. 11, 91058 Erlangen, Germany}
	\email{qianfeng.li@fau.de}
	\thanks{Qianfeng Li was partially supported by Sino--German (CSC--DAAD) Postdoc Scholarship Program, 2023 (No. 57678375). }
	
	\author{Yun Pu}
	\address{Academy of Mathematics and Systems Science, Chinese Academy of Sciences, Beijing 100190, China}
	\email{ypu@amss.ac.cn}
	\author{Yongqian Zhang}
	\address{School of Mathematical Sciences, Fudan University, Shanghai 200433, China.}
	\email{yongqianz@fudan.edu.cn}
	\thanks{Yongqian Zhang was partially supported by NSFC Project 11421061 and by NSFC Project 12271507.}
	
	
	\date{\today}
	
	\dedicatory{In Memory of Peter D. Lax} 
	

\begin{abstract}
	We analyze an inverse problem for determining the piston speed and the associated flow field from a prescribed leading 
	shock and the initial data in a shock tube. 
	The gas flow is described by the isentropic Euler equations ({\it i.e.}, the $p$-system), 
	while the trajectory of the leading shock is prescribed as a given Lipschitz curve. 
	Under an Ole{\u{i}}nik-type entropy condition on the leading shock, 
	we develop a modified wavefront tracking scheme to construct the flow field behind the shock. 
	This construction enables us to determine the corresponding piston speed and the associated flow field.
\end{abstract}
\maketitle

\section{Introduction}
When a piston moves into a gas initially at rest, a discontinuity, known as a shock wave, is immediately generated; 
see Fig.~\ref{fig:shock-tube}. 
This phenomenon is a fundamental feature of the shock-tube problem 
and plays a fundamental role in the general theory of nonlinear wave propagation in compressible fluids; 
see 
Courant-Friedrichs \cite{Courant1948} and Dafermos \cite{Dafermos2026}.

Considerable effort has been devoted to the corresponding direct problem of determining the flow field and leading shock-front 
from prescribed initial data and piston motion. 
For results concerning weighted $C^1$--perturbations, 
see \cite{li1994global,Wanglibin2007,Li1991GlobalShock,WangLibin2014}
and the references therein. 
For BV perturbations with prescribed piston speed, see \cite{amadori1997initial,Wang2005GlobalExistence,MR3582280} 
and the references therein. 
Results on $L^\infty$--solutions may be found in \cite{takeno1995free}
and the references therein.

\begin{figure}[ht]
\centering

\begin{subfigure}[b]{0.54\textwidth}    
\centering
\raisebox{0.6cm}{\includegraphics[width=\textwidth]{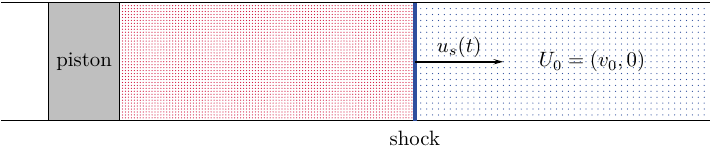}}
\caption{Shock tube}
\label{fig:shock-tube}
\end{subfigure}
\hfill
\begin{subfigure}[b]{0.44\textwidth}
\centering
\includegraphics[width=\textwidth]{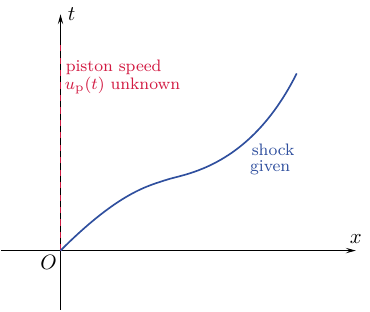}
\caption{Inverse piston problem}
\label{fig:inverse-piston-problem}
\end{subfigure}
\caption{Comparison of the shock tube and the inverse piston problem.}
\label{fig:combined-problems}
\end{figure}

In this paper, we analyze an inverse problem for determining the piston trajectory and the associated flow field 
from a prescribed leading shock and the initial data. 
In Lagrangian coordinates, the one-dimensional (1-D) isentropic gas dynamics is governed by the isentropic Euler equations, {\it i.e.}, 
the $p$-system:
\begin{equation}\label{eq:p-system}
\left\{  \begin{aligned}
\partial_tv-\partial_xu=0,\\
\partial_tu+\partial_xp=0,
\end{aligned}
\right.
\end{equation}
where $v>0$ is the specific volume, $u$ is the velocity, and $p=Av^{-\gamma}$ is the pressure, 
where $A>0$ is a constant and $\gamma\in(1,3)$ is the adiabatic constant.

System \eqref{eq:p-system} can be written in the form of conservation laws:
\begin{equation}\label{1.1a}
\partial_t U + \partial_x F(U)=0
\end{equation}
with $U=(v,u)^\top$ and $F(U)=(-u, p)^\top$.

The initial data for the gas at rest are given by 
\begin{equation}\label{eq:initial-condition}
U(x,0)=(v_0,0)^\top =:
U_0 \qquad \mbox{for $x>0$},
\end{equation}
where the specific volume $v_0>0$ is a given constant. 
The prescribed leading shock is given by 
\[
S\coloneqq \big\{(x,t)\,:\, x=\chi(t),\, t\geq 0\big\}
\]
with \(\chi(0)=0\) and \(\chi(\cdot) \in {\rm Lip}(\mathbb{R}_+)\).

The piston trajectory coincides with the positive $t$--axis, which is denoted by $\Gamma$.  
On the piston trajectory, the flow velocity coincides with the piston speed:
\begin{equation}\label{BoundaryC}
u(0,t)=u_{\rm p}(t).
\end{equation}
In this paper, we aim to determine 
the piston speed \(u_{\rm p}\) and the associated flow field  
$U=(v,u)$ in the region between $\Gamma$ and $S$,
given $\chi(\cdot)$ {\it a priori}; see Fig.~\ref{fig:inverse-piston-problem}. 

We seek entropy solutions of this inverse problem. 
\begin{definition}[Entropy Solutions]\label{def:entropy-solution}
A piston speed function $u_{\rm p}(t)\in\text{\rm BV}([0,\infty))$ 
and an associated flow field $U_{\rm e}=(v_{\rm e},u_{\rm e})^\top\in {\rm BV}_{\rm loc}(\mathbb{R}_+\times\mathbb{R}_+)$
form an entropy solution of the inverse problem \eqref{eq:p-system}--\eqref{BoundaryC}
if they satisfy the following{\rm :}
\begin{enumerate}
\item[\rm(i)] $U_{\rm e}$ is a global weak solution of \eqref{eq:p-system}, 
satisfying \eqref{eq:initial-condition}--\eqref{BoundaryC}
in the trace sense{\rm ;}

\smallskip
\item[\rm (ii)] For any 
	convex entropy $\eta\in C^1(\mathbb{R}^2;\mathbb{R})$ and the associated entropy flux $q\in C^1(\mathbb{R}^2;\mathbb{R})$,
satisfying $D\eta\, DF=Dq$,
the flow field $U_{\rm e}$ satisfies the Lax entropy inequality{\rm :}
\begin{equation}\label{eq:entropy}
	\partial_t\eta(U_{\rm e})+\partial_x q(U_{\rm e})\le 0
\end{equation}
in the distributional sense on $\mathbb{R}_+\times\mathbb{R}_+$; see Lax \cite{Lax_1973}.
\end{enumerate}
\end{definition}

We now state the main theorem:  
\begin{theorem}\label{thm:main}
Let $v_0>0$ and $s_0>c_0:=\sqrt{-p'(v_0)}$. Then 
\begin{enumerate}
\item[\rm (i)] \textbf{Existence of the profile}{\rm :} 
There exists a unique state  $\overline{U}=(\overline{v},\overline{u})^\top$ such that the piston speed $u_{\rm p}=\overline{u}$ 
and the vector function
\begin{align*}
	U_{\rm pro}(x,t)={\begin{cases}
			\overline{U}\,\, &\mbox{for $0<x<s_0\,t$},\\
			U_0\,\, &\mbox{for $x>s_0\,t$},
	\end{cases}}
\end{align*}
form an entropy solution.

\smallskip
\item[\rm (ii)] \textbf{Stability of the profile}{\rm :} 
Let \(s=\dot{\chi}\in {\rm BV}\cap L^\infty([0,\infty))\) be the speed of a leading shock such that \(s(t)+\qq(t)\) is monotonically increasing, 
where $\qq\in C^1([0,\infty))$ is a function satisfying 
\begin{equation}\label{eq:q-condition}
	0\le (1+t)\,\qq'(t)\le \Lambda \qquad\mbox{for $\Lambda>0$ sufficiently small}.
\end{equation}
Then there exists $\epsilon_{\rm inv}>0$ so that, for any such speed $s$ satisfying
\begin{equation}\label{eq:solution-bound}
	\|s-s_0\|_{L^\infty\cap BV}<\epsilon_{\rm inv},
\end{equation}
there exists a flow field \(U=(v,u)\) in the region between $\Gamma$ and $S$ such that the piston speed $u_{\rm p}=u(0,\cdot)$ 
and the vector function 
\begin{align*}
	U_{\rm e}(x,t)={\begin{cases}
			U(x,t)\,\, &\mbox{for $0<x<\chi(t)$},\\
			U_0\,\, &\mbox{for $x>\chi(t)$},
	\end{cases}}
\end{align*}
form an entropy solution to the inverse problem \eqref{eq:p-system}--\eqref{BoundaryC}.
Furthermore,
\begin{equation}\label{eq:stability}
	\|u_{\rm p}(\cdot)-\overline{u}\|_{L^\infty\cap BV}\le C\,\|s-s_0\|_{L^\infty\cap BV}
\end{equation}
for some $C>0$ depending only on $v_0$, $s_0$, and system \eqref{eq:p-system}. 
\end{enumerate}
\end{theorem}

\begin{remark}
    The assumption that \(s(t)+\qq(t)\) is monotonically increasing for some \(\qq\) satisfying \eqref{eq:q-condition} 
    is an Ole{\u\i}nik-type entropy condition. Technically, this condition ensures the strength of each backward 2-rarefaction wave 
    fronts tends to {\it zero}, which ensures the limit of the approximate solutions satisfies the entropy inequality; 
    see \cite[\S 7.1 and \S 7.4]{Bressan2000}. 
\end{remark}

For the $1$-D inverse piston problem, when the prescribed initial data are small perturbations of a constant state 
and the prescribed shock trajectory is close to a straight line, 
Li-Wang \cite{Wanglibin2007}, Wang \cite{WangLibin2014}, and Wang-Wang \cite{Wanglibin2019} 
employed the method of characteristics to determine the global piston speed and the corresponding global piecewise smooth flow. 
More recently, Hu-Li-Zhang \cite{hu2025inverse} studied the inverse piston problem in the hypersonic regime, 
allowing for substantially varying leading shock profiles.
Other inverse problems for hyperbolic conservation laws include the identification of initial data for the Burgers equation \cite{allahverdi2016numerical,castro2008alternating,castro2010optimal,colombo2020initial,colombo2023initial,colombo2024initial,colombo2025initial,esteve2020inverse,gosse2017filtered,liard2021initial,liard2023analysis}, 
the reconstruction of obstacle shapes in supersonic flows 
past obstacles \cite{CPZ-2025a,CPZ-2025b,hu2024inverse,li2022inverse,li2006global,wang2011direct}, as well as the characterization 
of attainable sets and stabilization for scalar conservation laws and hyperbolic systems under boundary 
controls \cite{ancona1998,ancona1999,ancona2005,ancona2007}.

In this paper, we develop a modified wavefront tracking algorithm 
to establish the global existence of entropy solutions of the inverse piston problem. 
Assuming that the total variation of the prescribed leading shock speed is sufficiently small, 
the entropy condition enables us to identify a family of lines
$\{\LL_{\tau}\}$ that are orthogonal to the time-marching direction $\bm n_{\LL}$; see \eqref{eq:space-time-direction}. 
We then solve a class of forward-backward Riemann-type problems in the $\bm n_{\LL}$--direction and 
construct approximate solutions up to the boundary: $x=0$. 
At the same time, the approximate piston velocity is determined through the non-reflection boundary condition.

Unlike the corresponding direct problem, the characteristic structure of the inverse problem is mixed: the $1$-characteristics propagate forward, whereas the $2$-characteristics propagate backward. 
In particular, backward waves emanating from the convex portions of the leading shock may exhibit a tendency to converge.
When such a case occurs, the strengths of certain backward rarefaction waves have a strictly positive lower bound. 
As a consequence, the entropy inequality may fail in the limiting process.
We emphasize that this phenomenon does not arise in \cite{yu2017}, where the backward characteristic family is assumed to be linearly degenerate, 
thereby excluding nonlinear focusing effects. 
Consequently, there is no analogous mechanism leading to a breakdown of entropy admissibility in that setting.
A related difficulty also appears in the inverse time-direction problem studied by Glass \cite{glass2007}, in which entropy-admissible solutions are obtained by preventing the intersections of backward characteristics. 
In \cite{glass2007}, this is achieved by freely adjusting two Riemann invariants to control the characteristic geometry. 
In the present problem, however, the situation is substantially more constrained. The two Riemann invariants behind the leading shock 
are coupled through the Rankine-Hugoniot conditions and are uniquely determined by the prescribed shock speed. 
Consequently, the available degrees of freedom are significantly restricted.

To overcome this difficulty, we introduce an Ole{\u\i}nik-type entropy condition requiring that $s+\qq$ 
be monotonically increasing, where $\qq$ is a suitably chosen slowly increasing function; see \eqref{eq:q-condition}. 
This condition imposes a restriction on the curvature of the convex portion of the leading shock profile, 
ensuring that the waves of the same family emanating from the leading shock do not intersect in the first quadrant.
As a consequence, the strength of every backward $2$-rarefaction front remains sufficiently small,
which allows us to continue the construction of approximate solutions in the $\bm n_{\LL}$--direction. 
Moreover, the limiting solution obtained from these approximations satisfies the Lax entropy inequality.

\smallskip
The remainder of this paper is organized as follows:
In \S\ref{sect-wave-curve-riemann-problem}, we review the wave curves associated with \eqref{eq:p-system} 
and analyze a class of forward-backward Riemann-type problems.
In \S\ref{sect-wave-front-tracking}, we approximate the prescribed leading shock by piecewise linear functions 
and construct a family of approximate solutions by using a modified wavefront tracking algorithm.
In \S\ref{sect-extension}, we introduce a Glimm-type functional and establish its monotonicity. 
We further show that the distance between any two wavefronts of the same family emanating from 
the leading shock remains positive in the first quadrant. 
As a consequence, the approximate solutions can be extended globally while maintaining uniform smallness of all rarefaction fronts.
Finally, in \S\ref{sect-proof-main-theorem}, we complete the proof of the main theorem.

\section{Wave Curves and Forward-Backward Riemann-Type Problems}\label{sect-wave-curve-riemann-problem}
In this section, we briefly recall some basic properties of the wave curves associated with 
system \eqref{eq:p-system} and study forward-backward Riemann-type problems. 
We also derive several wave interaction estimates that will be used throughout the paper.

\subsection{Wave curves}
As $p'(v)<0$, system \eqref{eq:p-system} is strictly hyperbolic, whose Jacobian matrix has two distinct eigenvalues
\begin{equation}\label{eq:eigenvalue}
\lambda_1=-\sqrt{-p'(v)},\qquad \lambda_2=\sqrt{-p'(v)},
\end{equation} 
with corresponding eigenvectors
\begin{equation*}
\bm r_1=\frac{2\sqrt{-p'(v)}}{p''(v)}(1,\sqrt{-p'(v)})^\top,\qquad \bm r_2=\frac{2\sqrt{-p'(v)}}{p''(v)}(-1,\sqrt{-p'(v)})^\top,
\end{equation*} 
satisfying
\begin{equation}\label{eq:genuinely-nonlinear}
\bm r_i\cdot\nabla\lambda_i=1\qquad\,\, \mbox{for $i=1,\,2$}.
\end{equation}

Therefore, the $i$-rarefaction curve through a state $U_-$ satisfies 
\begin{equation*}
\frac{\dd u}{\dd v}=(-1)^{i+1}\sqrt{-p'(v)},\qquad u(v_-)=u_-,
\end{equation*}
which gives 
\begin{equation*}
\bm R_i=\big\{(v,u)\,:\, u=u_-+(-1)^{i+1}\int_{v_-}^v\sqrt{-p'(y)}\dd y\big\}\qquad \mbox{for $i=1,\,2$}.
\end{equation*}

Meanwhile, the shock curves $\bm S_i$ through $U_-$ are determined from the Rankine-Hugoniot conditions:
\begin{equation}\label{eq:rh-condition}
\begin{aligned}
&s(v-v_-)={}-(u-u_-),\\
&s(u-u_-)={}p(v)-p(v_-).
\end{aligned} 
\end{equation}
By eliminating $s$ in \eqref{eq:rh-condition}, we obtain
\begin{align}\label{eq:shock-wave}
\begin{aligned}
&\bm S_1={}\big\{(v,u)\,:\, -(u-u_-)^2=(v-v_-)(p(v)-p(v_-)),\,s=-\frac{u-u_-}{v-v_-}<0\big\},\\
&\bm S_2={}\big\{(v,u)\,:\, -(u-u_-)^2=(v-v_-)(p(v)-p(v_-)),\,s=-\frac{u-u_-}{v-v_-}>0\big\}.
\end{aligned} 
\end{align}

\begin{figure}[ht]
\centering  
\includegraphics[width=0.75\textwidth]{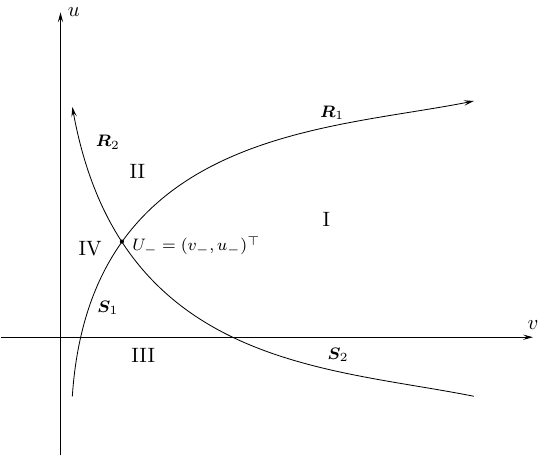}
\caption{Wave curves through $U_-$}
\label{fig:wave-curve}
\end{figure}

As pointed out in \cite[p.~105]{Bressan2000}, by \eqref{eq:genuinely-nonlinear}, the wave curves through $U_-$ (see Fig.~\ref{fig:wave-curve}) are
\begin{equation}\label{eq-wave-curve}
\begin{aligned}
\bm T_1={}&\left\{(v,u)\,:\,\, 
u=\left\{\begin{aligned}
	&u_-+\int_{v_-}^v\sqrt{-p'(y)}\dd y&\,\,\,\, \mbox{ for $v\ge{} v_-$},\\
	&u_--\sqrt{-(v-v_-)(p(v)-p(v_-))}\quad &\mbox{for $v<{} v_-$}
\end{aligned}
\right.\right\},\\[1mm]
\bm T_2={}&\left\{(v,u)\, :\,\, 
u=\left\{\begin{aligned}
	&u_--\int_{v_-}^v\sqrt{-p'(y)}\dd y &\,\,\,\,\mbox{ for $v\le{} v_-$},\\
	&u_--\sqrt{-(v-v_-)(p(v)-p(v_-))}\quad &\mbox{for $v>{}v_-$}
\end{aligned}
\right.\right\}.
\end{aligned} 
\end{equation}

\smallskip
\subsection{Strong shocks and the background solutions}
On the 2-shock, we know from \eqref{eq:shock-wave} that
\begin{equation}\label{eq:shock-speed}
s=\sqrt{-\frac{p(v)-p(v_-)}{v-v_-}},
\end{equation}
which implies that $\frac{\dd s}{\dd v_-}<0$, so that the strong $2$-shock can be parameterized by 
\begin{align*}
&s\mapsto G(s;\,U)\qquad\mbox{for $s>0$},\\
&G(s;\,U)=U_-.
\end{align*}
Then
we have 
\begin{lemma}\label{lem:profile}
Given $s_0>\sqrt{-p'(v_0)}$,
there exists a unique state $\overline{U}=(\overline{v},\overline{u})^\top$ connected with 
the initial state $U_0=(v_0,0)^\top$ by a $2$-shock $s_0$ in the form{\rm :}
\begin{equation*}
G(s_0;\,U_0)=\overline{U}.
\end{equation*}
\end{lemma}

\begin{proof}
By \eqref{eq:shock-speed}, it suffices to solve 
\begin{equation*}
s_0=\sqrt{-\frac{p(v_0)-p(\overline{v})}{v_0-\overline{v}}}
\end{equation*}
for $\overline{v}$, which is equivalent to 
\begin{equation*}
A \overline{v}^{-\gamma}+s_0^2\overline{v}-( A v_0^{-\gamma}+s_0^2v_0)=0.
\end{equation*}
Define $f(v)=A v^{-\gamma}+s_0^2v$.
Then $f'(v)= -\gamma A v^{-\gamma-1}+s_0^2$ has a unique zero $v_{\rm zero}$:
\begin{equation*}
v_{\rm zero}=\Big(\frac{s_0^2}{\gamma A}\Big)^{-\frac{1}{\gamma+1}}.
\end{equation*}
Since 
$s_0>\sqrt{-p'(v_0)}$, 
then $v_{\rm zero}<v_0$. 
Moreover,  $f(v)$ attains its minimum at $v=v_{\rm zero}$; see Table~\ref{tab:1}.
Thus, $f(v_{\rm zero})< A v_0^{-\gamma}+s_0^2v_0$, 
and there is a unique point $\overline{v}$ such that $f(\overline{v})= A v_0^{-\gamma}+s_0^2v_0$ on $(0,v_{\rm zero})$. 
This completes the proof.
\end{proof}
\begin{table}[ht] 
\renewcommand\arraystretch{1.5}
\noindent\[
\begin{array}{c|c|c|c}
&{(0,v_{\rm zero})}&{v_{\rm zero}}&{(v_{\rm zero},+\infty)}\\
\hline
{f'(v)}&-&0 &+\\
\hline
{f(v)}&\searrow&(\gamma+1)A(\frac{s_0^2}{\gamma A})^{\frac{\gamma}{\gamma+1}}&\nearrow\\
\end{array}
\]
\caption{The monotonicity of $f(v)$}\label{tab:1}
\end{table}

Moreover, we deduce the following lemma: 
\begin{lemma}\label{lem:derivative-strong-shock}
Denote $(v_-,u_-)=G(s;\,U_0)$. Then
\begin{equation*}
\frac{\dd G(s;\,U_0)}{\dd s}=\frac{v_0-v_-}{s^2+p'(v_-)}\,(2s,\,p'(v_-)-s^2)^\top.
\end{equation*}
\end{lemma}

\begin{proof}
Differentiating the Rankine–Hugoniot condition \eqref{eq:shock-speed} with respect to $s$, we obtain
\begin{equation*}
2s(v_0-v_-)-s^2(v_-)_s-p'(v_-)(v_-)_s=0,
\end{equation*}
which implies 
\begin{equation*}
(v_-)_s=\frac{2s(v_0-v_-)}{s^2+p'(v_-)}.
\end{equation*}
Recalling that $s=-\frac{u_0-u_-}{v_0-v_-}$, we obtain
\begin{equation*}
(u_-)_s=\frac{(p'(v_-)-s^2)(v_0-v_-)}{s^2+p'(v_-)},
\end{equation*}
which completes the proof.
\end{proof}

To proceed further, we first recall from \eqref{eq:eigenvalue} and \eqref{eq:shock-speed} that 
$\lambda_2(\overline{U})=-\lambda_1(\overline{U})=\sqrt{-p'(\overline{v})}$ and 
\begin{equation*}
s_0=\sqrt{-\frac{p(\overline{v})-p(v_0)}{\overline{v}-v_0}}=\sqrt{-p'(\tilde{v})}
\qquad \text{for some $\overline{v}<\tilde{v}<v_0$}.
\end{equation*}
Since $p''(v)>0$, we see that $\lambda_2(\overline{U})>s_0$. Define
\begin{equation*}
\lambda^\n=\frac{s_0+\lambda_2(\overline{U})}{2},\quad c=\frac{\lambda_2(\overline{U})-s_0}{4},
\quad \check{\lambda}=\frac{s_0}{4}+\frac{3\lambda_2(\overline{U})}{4}.
\end{equation*}
Then there exists $\varepsilon>0$ sufficiently small such that
\begin{equation}\label{eq:separate}
\lambda_2(U)>\lambda_2(\overline{U})-c=\check{\lambda}>\lambda^\n>s_0+c>s
\end{equation}
for $U\in O_\varepsilon(\overline{U})$ and $s\in O_\varepsilon(s_0)$.

Now, for $\tau\ge0$, set 
\begin{equation}\label{eq:space-time-direction}
\LL_\tau=\big\{(x,t)\,:\,x=\lambda^\n\,(t-\tau),\,t\ge\tau\big\},\qquad \bm n_{\LL}=(-1,\lambda^\n)^\top.
\end{equation}

\subsection{Weak waves and forward-backward Riemann-type problems}
For weak waves, we can parameterize two wave curves as 
\begin{align*}
T_1(\sigma)(U_-)=
\left\{\begin{aligned}
&R_1(\sigma)(U_-)& \mbox{for $\sigma\ge 0$},\\
&S_1(\sigma)(U_-)& \mbox{for $\sigma<0$},
\end{aligned}
\right.\qquad\,\,\,
T_2(\sigma)(U_-)=
\left\{\begin{aligned}
&R_2(\sigma)(U_-)& \mbox{for $\sigma\ge 0$},\\
&S_2(\sigma)(U_-)& \mbox{for $\sigma<0$},
\end{aligned}
\right.
\end{align*}
such that 
\begin{equation}\label{eq:parameterization}
\frac{\dd \lambda_i(T_i(\sigma)(U_-))}{\dd\sigma}=1\qquad \mbox{for $i=1,\,2$}.
\end{equation}
With the same parameterization, we define the inverse wave curves as 
\begin{align*}
\mathfrak{T}_1(\sigma)(U_-)=
\left\{\begin{aligned}
&S_1(\sigma)(U_-)& \mbox{for $\sigma\ge 0$},\\
&R_1(\sigma)(U_-)& \mbox{for $\sigma<0$},
\end{aligned}
\right.\qquad
\mathfrak{T}_2(\sigma)(U_-)=
\left\{\begin{aligned}
&S_2(\sigma)(U_-)& \mbox{for $\sigma\ge 0$},\\
&R_2(\sigma)(U_-)& \mbox{for $\sigma<0$}.
\end{aligned}
\right.
\end{align*}
Now, denote 
\begin{align*}
\Phi_i(\alpha_i;\,U_-)=T_i(\alpha_i)(U_-),\quad
\Xi_i(\alpha_i;\,U_-)=\mathfrak{T}_i(\alpha_i)(U_-)
\qquad\,\,\mbox{for $i=1,\,2$}. 
\end{align*}
It follows from \eqref{eq:parameterization} that 
\begin{equation}\label{eq:invertible}
U_-=\Xi_i(-\alpha_i;\,\Phi_i(\alpha_i;\,U_-))\qquad\mbox{for $i=1,\,2$},
\end{equation} 
and 
\begin{equation}\label{eq:wave-curve-parameterization-derivative}
\left.\frac{\dd \Phi_i(\alpha_i;\,U_-)}{\dd\alpha_i}\right|_{\alpha_i=0}
= \left.\frac{\dd \Xi_i(\alpha_i;\,U_-)}{\dd\alpha_i}\right|_{\alpha_i=0}=\bm r_i(U_-)
\qquad\mbox{for $i=1,\,2$}.
\end{equation}
Let
\begin{equation*}
\Phi(\bm \alpha;\,U_-)=\Phi_2(\alpha_2;\,\Phi_1(\alpha_1;\,U_-))
\end{equation*}
be the composition of two families of wave curves with $\bm \alpha=(\alpha_1,\,\alpha_2)$. 

Next, we solve forward-backward Riemann-type problems. 
By a forward-backward Riemann-type problem, we mean the Cauchy problem with initial data of the form:
\begin{align}\label{eq:inverse-riemann-data}
U(x,t)=\left\{
\begin{aligned}
&U_b=(v_b,u_b)^\top \quad& \mbox{for $x=\lambda^\n(t-\bar{t})+\bar{x}$ and $t<\bar{t}$},\\
&U_a=(v_a,u_a)^\top \quad& \mbox{for $x=\lambda^\n(t-\bar{t})+\bar{x}$ and $t>\bar{t}$},
\end{aligned}
\right.
\end{align}
whose solvability is guaranteed by the following lemma:
\begin{figure}[ht]
\centering
\includegraphics[width=0.45\textwidth]{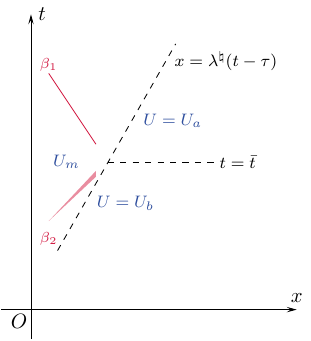}             
\caption{Forward-backward Riemann-type problems} 
\label{fig:forward-backward-riemann-problem}
\end{figure}

\begin{lemma}\label{lem:inverse-riemann-problem}
There exists $\epsilon_0$ sufficiently small such that, if $U_a$, $U_b\in O_{\epsilon_0}(\overline{U})$, then the equation{\rm :} 
\begin{equation}\label{eq:inverse-riemann-problem}
\Phi_1(\beta_1;\,\Xi_2(-\beta_2;\,U_b))=U_a
\end{equation}
has a unique solution $\bm \beta=(\beta_1,\beta_2)$ with
\begin{equation}\label{eq:wave-strength-control}
|\beta_1|+|\beta_2|\le O(1)|U_a-U_b|,
\end{equation}
where $O(1)$ is a bounded quantity depending only on $\overline{U}$ and system \eqref{eq:p-system}{\rm ;}
see {\rm Fig.~\ref{fig:forward-backward-riemann-problem}}.
\end{lemma}

\begin{proof}
Differentiating both sides of equation \eqref{eq:inverse-riemann-problem}
with respect to $\bm \beta$ and using \eqref{eq:invertible}--\eqref{eq:wave-curve-parameterization-derivative}, we obtain
\begin{equation*}
\left.\frac{\partial}{\partial \bm \beta}\Phi_1(\beta_1;\,\Xi_2(-\beta_2;\,U_b))\right|_{U_b=\overline{U},\,\bm \beta=\bm 0}=(\bm r_1(\overline{U}),\,-\bm r_2(\overline{U})),
\end{equation*}
which is nonsingular. 
Then the implicit function theorem (see, for example, \cite[Theorem 2.2]{Bressan2000}) 
implies the existence of $\bm \beta$. 
Finally, \eqref{eq:wave-strength-control} is guaranteed by the fact that 
\begin{equation*}
U_a-U_b=-\beta_2\,\int_0^1\frac{\partial \Xi_2}{\partial\beta_2}(-\theta\beta_2;\,U_b)\,\dd\theta
+\beta_1\,\int_0^1\frac{\partial \Phi_1}{\partial\beta_1}(\theta\beta_1;\,\Xi_2(-\beta_2;\;U_b))\,\dd\theta.
\end{equation*}
This completes the proof.
\end{proof}

\subsection{Wave interaction estimates}
In this subsection, we derive some wave interaction estimates when solving the forward-backward Riemann-type problems.
\begin{figure}[ht]
\centering 
\includegraphics[width=0.45\textwidth]{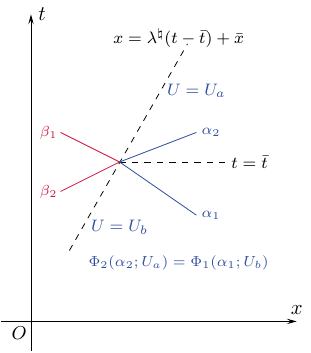} 
\caption{Wave interaction} 
\label{fig:wave-interaction}
\end{figure}

\begin{lemma}\label{lem:wave-interaction-estimate-1}
Suppose that $U_a$, $U_b\in O_{\epsilon_0}(\overline{U})$ with 
\begin{equation*}
\Phi_1(\alpha_1;\;U_b)=\Phi_2(\alpha_2;\;U_a).
\end{equation*}
Then there exists $\bm \beta=(\beta_1,\beta_2)$ such that
\begin{equation*}
\Phi_1(\beta_1;\;\Xi_2(-\beta_2;\,U_b))=U_a
\end{equation*}
with 
\begin{equation*}
\beta_i=\alpha_i+O(1)|\alpha_1||\alpha_2|\qquad\mbox{for $i=1,\,2$},
\end{equation*}
where $O(1)$ is a bounded quantity depending only on $\overline{U}$ and the system \eqref{eq:p-system}{\rm ;}
see {\rm Fig.~\ref{fig:wave-interaction}}.
\end{lemma}

\begin{proof}
The existence of $\bm \beta$
follows from Lemma \ref{lem:inverse-riemann-problem}. Furthermore, the assumption and \eqref{eq:invertible} imply that
\begin{equation*}
\Xi_2(-\alpha_2;\,\Phi_1(\alpha_1;\,U_b))=U_a,
\end{equation*}
which leads to 
\begin{equation*}
\Xi_2(-\alpha_2;\,\Phi_1(\alpha_1;\,U_b))-\Phi_1(\beta_1;\,\Xi_2(-\beta_2;\,U_b))=\bm 0.
\end{equation*}
From the proof of Lemma \ref{lem:inverse-riemann-problem}, there is a unique $\bm \beta$ such that
\begin{equation*}
\beta_i=\beta_i(\bm\alpha)\qquad\mbox{for $i=1,\,2$}.
\end{equation*} 
In particular, we have
\begin{align*}
\beta_1(0,\alpha_2)={}&0,&\beta_2(0,\alpha_2)={}&\alpha_2,\\
\beta_1(\alpha_1,0)={}&\alpha_1,&\beta_2(\alpha_1,0)={}&0.
\end{align*} 
Thus, by Lemma \ref{lem:quadratic}, we obtain
\begin{align*}
\beta_1={}&\beta_1(0,\alpha_2)+\beta_1(\alpha_1,0)-\beta_1(0,0)+O(1)|\alpha_1||\alpha_2|\\
={}&\alpha_1+O(1)|\alpha_1||\alpha_2|,\\
\beta_2={}&\beta_2(0,\alpha_2)+\beta_2(\alpha_1,0)-\beta_2(0,0)+O(1)|\alpha_1||\alpha_2|\\
={}&\alpha_2+O(1)|\alpha_1||\alpha_2|.
\end{align*} 
This completes the proof.
\end{proof}

The state behind the leading shock is controlled by the shock speed, which is summarized as the following lemma:  

\begin{figure}[ht]
\centering
\begin{subfigure}[b]{0.45\textwidth} 
\centering
\includegraphics[width=\textwidth]{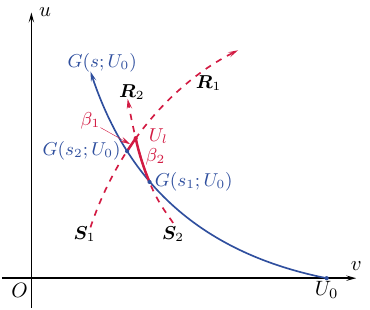} 
\caption{Case 1: $s_2>s_1$} 
\label{fig:strong-shock-interaction-1}
\end{subfigure}
\hfill 
\begin{subfigure}[b]{0.45\textwidth} 
\centering
\includegraphics[width=\textwidth]{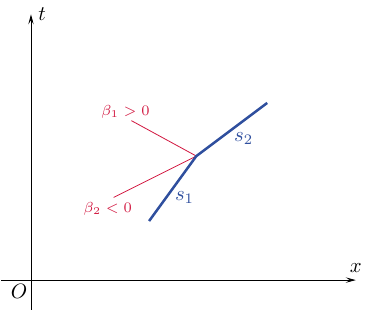} 
\caption{Case 1: $s_2>s_1$} 
\label{fig:strong-shock-interaction-3}
\end{subfigure}
\vspace{1cm}
\begin{subfigure}[b]{0.45\textwidth}
\centering
\includegraphics[width=\textwidth]{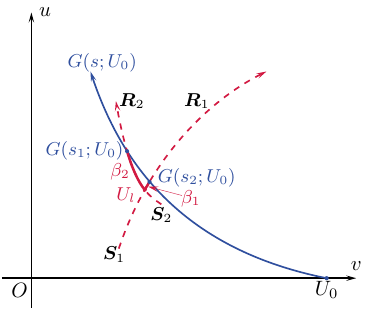} 
\caption{Case 2: $s_2<s_1$} 
\label{fig:strong-shock-interaction-2}
\end{subfigure}
\hfill 
\begin{subfigure}[b]{0.45\textwidth} 
\centering
\includegraphics[width=\textwidth]{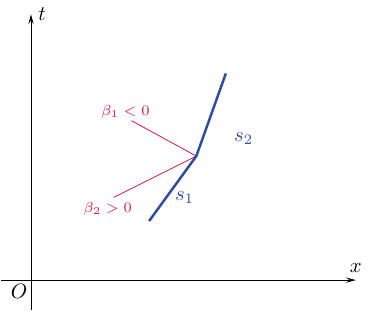} 
\caption{Case 2: $s_2<s_1$} 
\label{fig:strong-shock-interaction-4}
\end{subfigure}
\caption{Leading shock generates weak waves} 
\end{figure}

\begin{lemma}\label{lem:strong-wave-interaction-estimate}
There exists $\epsilon_s>0$ such that, for $s_1$, $s_2\in O_{\epsilon_s}(s_0)$,
\begin{align*}
&G(s_1;\,U_0),\,G(s_2;\,U_0)\in O_{\epsilon_0}(\overline{U}),\\[1mm]
& |G(s_1;\,U_0)-G(s_2;\,U_0)|= O(1)|s_1-s_2|,
\end{align*}
where $O(1)$ is a bounded quantity depending only on $\overline{U}$ and the system \eqref{eq:p-system}. 
Moreover, denoting $U_l=\Xi_2(-\beta_2;\,G(s_1;\,U_0))$, there exists $\bm \beta=(\beta_1,\beta_2)$ such that
\begin{equation}\label{eq:strong-wave-interaction-estimate}
\Phi_1(\beta_1;\,U_l)=G(s_2;\,U_0),
\end{equation}
and, in particular, 
\begin{enumerate}
\item[\rm (i)]  $\beta_1>0$ and $\beta_2<0$  when $s_1<s_2$ {\rm (}see {\rm Fig.~\ref{fig:strong-shock-interaction-1}--\ref{fig:strong-shock-interaction-3}}{\rm );}
\item[\rm (ii)]   $\beta_1<0$ and $\beta_2>0$ when $s_1>s_2$ {\rm (}see {\rm Fig.~\ref{fig:strong-shock-interaction-2}--\ref{fig:strong-shock-interaction-4}}{\rm ).}
\end{enumerate}
\end{lemma}

\begin{proof}
The proof of the first part is direct, so we omit its details here. 
For the second part, taking the derivative of \eqref{eq:strong-wave-interaction-estimate}
with respect to $s_2$ 
and evaluating at $s_1=s_2=s_0$, we obtain
\begin{equation*}
\frac{\partial\beta_1}{\partial s_2}{\bm r}_1(\overline{U})-\frac{\partial\beta_2}{\partial s_2}{\bm r}_2(\overline{U})= G'(s_0;\,U_0).
\end{equation*}
The mean value theorem implies that $s_0=\sqrt{-p'(\zeta)}$ for $\overline{v}<\zeta<v_0$. 
Then Cramer's rule gives
\begin{align*}
\left. \frac{\partial\beta_1}{\partial s_2}\right|_{s_2=s_0}={}&\frac{(\overline{v}-v_0)p''(\overline{v})\big(\sqrt{-p'(\overline{v})}
-\sqrt{-p'(\zeta)}\big)}{4p'(\overline{v})\big(\sqrt{-p'(\overline{v})}+\sqrt{-p'(\zeta)}\big)}>0,\\
\left. \frac{\partial\beta_2}{\partial s_2}\right|_{s_2=s_0}={}&-\frac{(\overline{v}-v_0)p''(\overline{v})\big(\sqrt{-p'(\overline{v})}
+\sqrt{-p'(\zeta)}\big)}{4p'(\overline{v})\big(\sqrt{-p'(\overline{v})}-\sqrt{-p'(\zeta)}\big)}<0.
\end{align*}
This completes the proof, provided that $\epsilon_s$ is sufficiently small.
\end{proof}

\section{A Modified Wavefront Tracking Algorithm}\label{sect-wave-front-tracking}
In this section, we develop a modified wavefront tracking algorithm and construct approximate solutions 
for the inverse problem \eqref{eq:p-system}--\eqref{BoundaryC}. 

To construct the approximate solution $U^{\mu,\Delta t}(x,t)$ and the approximate piston speed $u_{\rm p}^{\mu,\Delta t}$ in the $\bm n_{\LL}$--direction, 
we solve the forward-backward Riemann-type problems and apply the boundary condition \eqref{BoundaryC} iteratively.

Define
\begin{equation*}
B(\bm n_{\LL},r;\,(\bar{x},\bar{t})):=O_r((\bar x,\bar t))\cap\{(x,t)\,:\,(x-\bar x,t-\bar t)\cdot\bm n_{\LL}>0\}.
\end{equation*}
We refer to $B(\bm n_{\LL},r;\,(\bar{x},\bar{t}))$ as the $\bm n_{\LL}$-forward neighborhood of $(\bar x,\bar t)$, 
where $\bm n_{\LL}$ is defined in \eqref{eq:space-time-direction}.

The approximate solutions are piecewise constant vector-valued functions whose discontinuities are
represented by wavefronts. 
Each wavefront propagates in the $\bm n_{\LL}$--direction until it interacts either with another wavefront or 
with the approximate piston trajectory $\{(0,t)\,:\,t\geq 0\}$. At each interaction point, a new forward-backward Riemann-type problem 
is generated.

To resolve these subsequent Riemann problems, we adopt the accurate Riemann solvers ({\it cf.} \cite{Bressan2000}). 
These solvers accommodate both shock waves and rarefaction fronts.

Let $\delta = \delta(\Delta t) > 0$ denote a parameter (to be specified later) exceeding the maximum rarefaction front strength. 

Assume that $\Phi_1(\beta_1;\,\Xi_2(-\beta_2;\,U_b))=U_a$ with $\beta_2<\delta$. 
Then, in an $\bm n_{\LL}$-forward neighborhood of $(\bar x,\bar t)$, the approximate solution to the forward-backward Riemann-type 
problem with initial data \eqref{eq:inverse-riemann-data} has the following explicit formula:
\begin{align}
U^{\delta}_{\rm Rie}(x,t)=\left\{\begin{aligned}
&U_{m_0}\ &&\mbox{for $\lambda^\n<\xi<\sigma_1$},\\
&U_{m_1}\ &&\mbox{for $\sigma_1<\xi$},\\
&U_{m_1}\ &&\mbox{for $\xi<\sigma_2^-$},\\
&\Phi_1\left(\xi-\lambda_1(U_{m_1});\,U_{m_1}\right)\ &&\mbox{for $\sigma_2^-<\xi<\sigma_2^+$},\\
&U_{m_2}\ &&\mbox{for $\sigma_2^+<\xi<\lambda^\n$},
\end{aligned}
\right.\label{eq:self-similar-riemann-solution}
\end{align}
where
\begin{equation*}
\xi=\frac{x-\bar x}{t-\bar t},\quad U_{m_0}=U_b,\quad U_{m_1}=\Xi_2(-\beta_2;\,U_b),\quad
U_{m_2}=U_a,
\end{equation*}
and 
\begin{align*}
\sigma_1={}&\left\{
\begin{aligned}
&\,\text{the speed of the weak shock $\beta_2$ $\,\,\,\,$ when $\beta_2<0$},\\
&\,\lambda_2(U_{m_0}) \qquad\qquad\qquad\qquad\qquad\quad \text{ when $\beta_2>0$},
\end{aligned}
\right.&&\\
\sigma_2^+={}&\sigma_2^-=\sigma_2,\quad \sigma_2\text{ is the speed of the weak shock $\beta_1$ $\,\,\,\,$
when $\beta_1<0$},&&\\
\sigma_2^-={}&\lambda_1(U_{m_{1}}),\quad \sigma_2^+=\lambda_1(U_{m_{2}})
\qquad\qquad\qquad\qquad\qquad\quad\, \text{when $\beta_1>0$}.&&
\end{align*}

\subsection{Accurate Riemann solver} In this subsection, we introduce the accurate Riemann solvers, 
which are used as the building blocks to construct approximate solutions.

The accurate Riemann solver is an approximate solution of the forward-backward Riemann-type problem \eqref{eq:inverse-riemann-data}, 
in which each rarefaction wave is partitioned into piecewise constant states separated by rarefaction fronts. 
Specifically, consider a $1$-wave $\beta_1$, originating at $(\bar x,\bar t)$, 
that connects two constant states $U_{m_1}$ and $U_{m_2}$. 
Taking the smallest $n\in\mathbb{N}_+$ such that $\beta_1 < n\delta$, we then define
\begin{equation}\label{eq-accurate-riemann-solver}
U^\delta_{1,\beta_1}(x,t)\triangleq
U_{m_{1}}+\sum_{j=1}^n(U_{m_{1},j}-U_{m_{1},j-1})H(x-\bar x-(t-\bar t)\,\lambda_1(U_{m_{1},j-1})),
\end{equation}
where  $H$ is the Heaviside function and
\begin{equation*}
U_{m_{1},j}=\Phi_1(\frac{j\beta_1}{n};\,U_{m_1})\qquad \mbox{for $j=0,1,\cdots,n$};
\end{equation*}
see Fig.~\ref{fig:accurate-riemann-solver}.
\begin{figure}[ht]
\centering
\begin{subfigure}[b]{0.45\textwidth} 
\centering
\includegraphics[width=\textwidth]{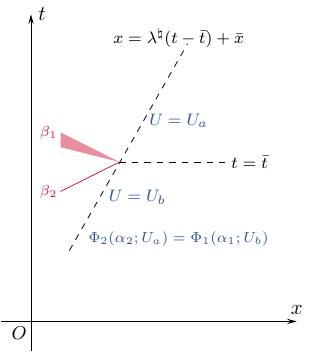}
\caption{Riemann solver} 
\end{subfigure}
\hfill 
\begin{subfigure}[b]{0.45\textwidth} 
\centering
\includegraphics[width=\textwidth]{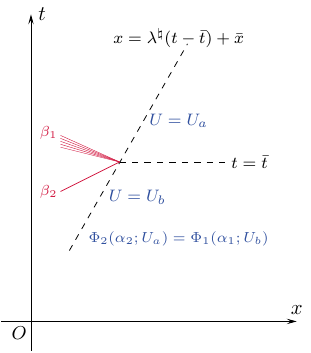}
\caption{Accurate Riemann solver} 
\end{subfigure}
\caption{Accurate Riemann solver} 
\label{fig:accurate-riemann-solver}
\end{figure}

To obtain an accurate Riemann solver $U^\delta_{\rm A}(U_{b},U_{a};\,x,t)$,
we use $U^\delta_{1,\beta_1}(x,t)$ to substitute
$\Phi_1\left(\xi-\lambda_1(U_{m_1});\,U_{m_1}\right)$ in the approximate 
solution \eqref{eq:self-similar-riemann-solution} when $(x,t)$ belongs to the rarefaction region
$\sigma_2^-(t-\bar t)<x-\bar x<\sigma_2^+(t-\bar t)$.

\subsection{Construction of the approximate solutions}
We now define the approximate solutions inductively. 
For simplicity of notation, in this section, we omit the superscript $\delta$ 
and assume that $\|s(\cdot)-s_0\|_{L^{\infty} \cap BV}$ is sufficiently small 
to satisfy the conditions of Lemma \ref{lem:strong-wave-interaction-estimate}. 

We first approximate the leading shock by piecewise linear functions. 
For any $\Delta t>0$, there exists $N_{\Delta t}>0$ such that 
\begin{equation*}
\text{TV}\{s\,:\, [N_{\Delta t}\,\Delta t,\,\infty)\}<\Delta t.
\end{equation*}

\begin{figure}[ht]
\centering  
\includegraphics[width=0.6\textwidth]{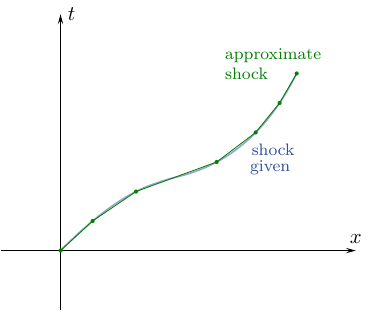}
\caption{Leading shock approximation}
\label{fig:leading-shock-approximation}
\end{figure}

We connect the points $\{(\chi(h\Delta t), h\Delta t)\,:\,\,h=0,\,1,\,2,\,\cdots,\,N_{\Delta t}\}$ by line segments 
to obtain a piecewise linear function $\chi^{\Delta t}$; 
see Fig.~\ref{fig:leading-shock-approximation}. 
Then the speed of the approximate leading shock is
\begin{equation}\label{eq-approximate-shock-speed}
\begin{aligned}
s^{\Delta t}(t)={}&s_{h}^{\Delta t}&&\,\,\,\mbox{for $t\in((h-1)\Delta t,\,h\Delta t]$ and $h=1,\,2,\,\cdots,\,N_{\Delta t}$},\\
s^{\Delta t}(t)={}&s(N_{\Delta t}\Delta t)&&\,\,\,\mbox{for $t\in (N_{\Delta t}\Delta t,\,\infty)$},
\end{aligned}  
\end{equation}
where 
\begin{align*}
s_{h}^{\Delta t}=\left\{
\begin{aligned}
&\frac{1}{\Delta t}\int_{(h-1)\Delta t}^{h\Delta t}s(r)\,\dd r\qquad \mbox{for $h=1,\,2,\,\cdots,\,N_{\Delta t}$},\\
&s(N_{\Delta t}\, \Delta t)\qquad\qquad\quad\,\,\,\,\,\mbox{for $h>N_{\Delta t}$}.
\end{aligned}\right.
\end{align*}

\begin{lemma}\label{lem:approximate-leading-shock}
The approximate shock speed constructed above has the following properties{\rm :} 
\begin{itemize}
\item [{\rm (i)}] $\displaystyle TV(s^{\Delta t})\le TV(s)${\rm ;}
\item [{\rm (ii)}] $\displaystyle s_{h+1}^{\Delta t}-s_h^{\Delta t}\geq-\frac{\Lambda \Delta t}{1+(h-1)\Delta t}$, 
$h=1,\,2,\,\cdots,\,N_{\Delta t}${\rm ;}
\item [{\rm (iii)}] For any fixed $T>0$, $\displaystyle \lim_{\Delta t\to 0}\int_0^T |s^{\Delta t}(t)-s(t)|\,\dd t=0$.
\end{itemize}
\end{lemma}

\begin{proof} Let 
\begin{equation*}
\tilde{s}(t)=\left\{
\begin{aligned}
&s(t)  &&\mbox{for $t\in[0,\,N_{\Delta t}\, \Delta t]$},\\
&s(N_{\Delta t}\,\Delta t)\quad &&\mbox{for $t\in(N_{\Delta t}\, \Delta t,\,\infty)$}.
\end{aligned}\right.
\end{equation*}
Then we have
\begin{align*}
TV(s^{\Delta t})={}&\sum_{h=1}^{N_{\Delta t}}\Big|\frac{1}{\Delta t}\int_{(h-1)\Delta t}^{h\Delta t}
\big(\tilde{s}(r+\Delta t)-\tilde{s}(r)\big)\,\dd r\Big|\\
\le{}&\frac{1}{\Delta t}\int_{0}^{\infty}|\tilde{s}(r+\Delta t)-\tilde{s}(r)|\,\dd r\\
\le{}&TV(s).
\end{align*}
Then {\rm (i)} follows. 

Moreover, the assumptions on $s$ imply
\begin{equation*}
s(r_1)-s(r_2)\geq \qq(r_2)-\qq(r_1)\qquad\mbox{for }r_1>r_2,
\end{equation*}
which, together with \eqref{eq:q-condition}, yields
\begin{align*}
s_{h+1}^{\Delta t}-s_h^{\Delta t}={}&
\frac{1}{\Delta t}\int_{(h-1)\Delta t}^{h\Delta t}\big(s(r+\Delta t)-s(r)\big)\,\dd r\\
\geq{}&\frac{1}{\Delta t}\int_{(h-1)\Delta t}^{h\Delta t}\big(\qq(r)-\qq(r+\Delta t)\big)\,\dd r\\
\geq{}&-\frac{\Lambda \Delta t}{1+(h-1)\Delta t}\qquad\mbox{for  }h=1,\,2,\,\cdots,\,N_{\Delta t}-1,\\
s_{N_{\Delta t}+1}^{\Delta t}-s_{N_{\Delta t}}^{\Delta t}={}
&\frac{1}{\Delta t}\int_{(N_{\Delta t}-1)\Delta t}^{N_{\Delta t}\Delta t}\big(s(N_{\Delta t}\Delta t)-s(r)\big)\,\dd r\\
\geq{}&\frac{1}{\Delta t}\int_{(N_{\Delta t}-1)\Delta t}^{N_{\Delta t}\Delta t}\big(\qq(r)-\qq(N_{\Delta t}\Delta t)\big)\,\dd r\\
\geq{}&-\frac{\Lambda \Delta t}{1+(N_{\Delta t}-1)\Delta t}\qquad\mbox{for  }h=N_{\Delta t},
\end{align*}
leading to {\rm (ii)}.

Finally, by the Lebesgue differentiation theorem, 
\begin{equation*}
\lim_{\Delta t\to 0}s^{\Delta t}(t)=s(t)\qquad a.e.
\end{equation*}
Since $\|s(\cdot)-s_0\|_{L^{\infty} \cap BV}$ is bounded, by the dominated convergence theorem, 
we conclude
\begin{equation*}
\lim_{\Delta t\to 0}\|s^{\Delta t}(\cdot)-s(\cdot)\|_{L^1([0,T])}=0.
\end{equation*}
The proof is complete.
\end{proof}

By Lemma \ref{lem:approximate-leading-shock}, for any $\mu>0$, we may approximate the speed of 
the leading shock by $s^{\mu,\Delta t}$ satisfying properties {\rm (i)}--{\rm (ii)} 
and $ \|s^{\mu,\Delta t}(\cdot)-s(\cdot)\|_{L^1([0,T])}<\mu$. 
Hence, the approximate leading shock trajectory is given by $\chi^{\mu,\Delta t}(t):=\int_0^t s^{\mu,\Delta t}(r)\,\dd r$.

Recalling \eqref{eq:separate}, it can be seen that, as $\|s(\cdot)-s_0\|_{L^{\infty}}$ is sufficiently small, for each $t$, 
there is a unique $\tau_t$ such that $\chi^{\mu,\Delta t}(t)=\lambda^\n\,(t-\tau_t)$, 
where $\tau_t$ is the {\it time} at which $\LL_{\tau_t}$ intersects the $t$-axis.

\begin{figure}[ht]
\centering  
\includegraphics[width=0.60\textwidth]{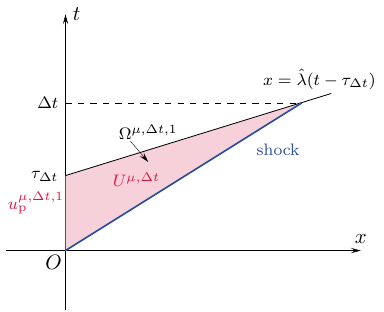}
\caption{The approximate solution near the origin}
\label{fig:approximate-solution-origin}
\end{figure}

When $h=1$, given $s^{\mu,\Delta t}_{1}\in O_{\epsilon_s}(s_0)$, 
according to Lemma \ref{lem:strong-wave-interaction-estimate}, we define
\begin{equation*}
u_{\rm p}^{\mu,\Delta t,1}:= G^{(2)}(s^{\mu,\Delta t}_1;\,U_0),
\end{equation*} 
where $G^{(k)}$ is the $k$-th component of $G$ for $k=1,\,2$.

Let $u_{\rm p}^{\mu,\Delta t}(t)\equiv u_{\rm p}^{\mu,\Delta t,1}$ for $0< t \leq\tau_{\Delta t}$,
and let
\begin{equation*}
U^{\mu,\Delta t}(x,t)=G(s^{\mu,\Delta t}_1;\,U_0)\qquad\text{for $(x,t)\in \Omega^{\mu,\Delta t,1}$},
\end{equation*} 
where 
\begin{equation*}
\Omega^{\mu,\Delta t,1}:=\big\{(x,t)\,:\,\,x>0,\,x<\chi^{\mu,\Delta t}(t),\, x\geq\lambda^\n\,(t-\tau_{\Delta t}),\, 0< t\leq \Delta t\big\};
\end{equation*} 
see Fig.~\ref{fig:approximate-solution-origin}.

Suppose now that the approximate piston speed $u_{\rm p}^{\mu,\Delta t}$ is defined for $t\in[0,\tau_{h\Delta t})$ 
and the approximate solution $U^{\mu,\Delta t}(x,t)$ is constructed on 
\begin{equation*}
\bigcup_{k=1}^{h}\Omega^{\mu,\Delta t,k}\qquad \text{for $h\in \mathbb{N}_+$},
\end{equation*}
we extend the approximate piston speed $u_{\rm p}^{\mu,\Delta t}$ to $t\in[\tau_{h\Delta t},\tau_{(h+1)\Delta t})$ and the approximate solution to 
\begin{align*}
\Omega^{\mu,\Delta t,h+1}={}&\big\{(x,t)\,:\, x>0,\,0< t< (h+1)\Delta t\big\}\\
&\cap\big\{(x,t)\,:\,\lambda^\n\,(t-\tau_{(h+1)\Delta t})<x\leq\lambda^\n\,(t-\tau_{h\Delta t}),\,0< t< (h+1)\Delta t\big\}\\
&\cap\big\{(x,t)\,:\,x<\chi^{\mu,\Delta t}(t),\,0< t< (h+1)\Delta t\big\},
\end{align*}
where $\tau_{h\Delta t}$ is the time at which line $\LL_{\tau_{h\Delta t}}$, passing through $(\chi^{\mu,\Delta t}(h\Delta t),h\Delta t)$, 
intersects the $t$-axis.

Denote
\begin{equation*}
\varepsilon_0=\min\{\varepsilon,\,\epsilon_0\},
\end{equation*}
where $\varepsilon$ and $\epsilon_0$ are introduced in \eqref{eq:separate} and Lemma \ref{lem:inverse-riemann-problem}, respectively.
Moreover, we assume that $u_{\rm p}^{\mu,\Delta t}$ and $U^{\mu,\Delta t}(x,t)$ satisfy the following:

\smallskip
\begin{itemize}
\item[$\mathrm{H}_1(h)$:] The approximate solution is piecewise constant vectors such that
$U^{\mu,\Delta t}\in O_{\varepsilon_0}(\overline{U})$ and, in particular, 
the approximate speed $u_{\rm p}^{\mu,\Delta t}(t)$ is piecewise constant with
$u_{\rm p}^{\mu,\Delta t}\in O_{\varepsilon_0}(\overline{u})$.

\vspace{2pt}  
\item[$\mathrm{H}_2(h)$:] The wavefronts of the same family do not intersect in \(\bigcup_{k=1}^{h}\Omega^{\mu,\Delta t,k}\).

\vspace{2pt}
\item[$\mathrm{H}_3(h)$:] There are finitely many jump discontinuities classified as rarefaction fronts $\mathcal{R}$ 
and weak shock fronts $\mathcal{S}$ with total wavefront count denoted by $N_h$. 
\end{itemize}

Slightly abusing notation, we denote a front by $\alpha$ and its strength by $|\alpha|$, 
and apply a modified wavefront tracking method to extend the approximate solutions; 
see Fig.~\ref{fig:wave-front-tracking}.
\begin{figure}[ht]
\centering  
\includegraphics[width=0.9\textwidth]{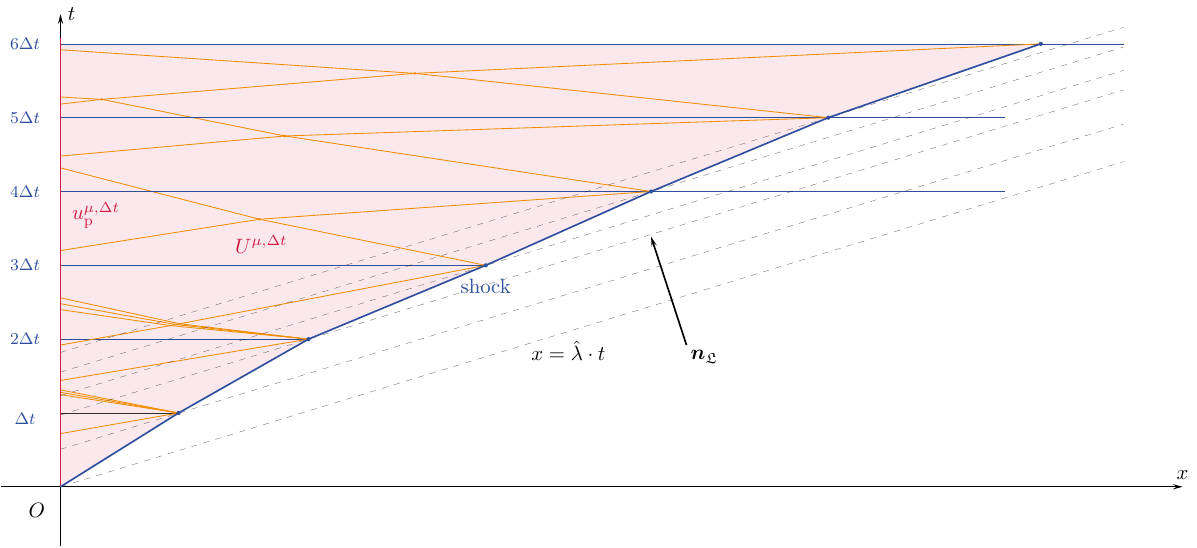}
\caption{Wavefront tracking}
\label{fig:wave-front-tracking}
\end{figure}
Wavefronts are typically propagated until they collide. 
When two wavefronts collide at a point $(\bar{x},\bar{t})\in\Omega^{\mu,\Delta t,h+1}$, 
they separate three states, denoted $U_b$, $U_m$, and $U_a$ (ordered from bottom to top).  
We then employ the accurate Riemann solver to continue the wavefront tracking construction; see also \cite{Bressan2000}.\par

\smallskip
\textit{Case 1. Weak waves are generated from the leading shock at $(\chi^{\mu,\Delta t}(k\Delta t),k\Delta t)$}: 
An accurate Riemann solver is applied.

\smallskip
\textit{Case 2. $\alpha$ and $\varpi$ are weak waves of different families colliding at $(\bar{x},\bar{t})$}: 
We solve the forward-backward Riemann-type problem with initial data \eqref{eq:inverse-riemann-data} 
by using the accurate Riemann solver. 
Moreover, the newly generated rarefaction fronts are never partitioned, regardless of their strength (even if their strength exceeds $\delta$).

\smallskip
\textit{Case 3. A weak wave $\alpha$ separating states $U_a$ and $U_b$ hits the boundary}: 
Let $\alpha$ cross the boundary and update $u_{\rm p}^{\mu,\Delta t}$ from $U_b^{(2)}$ to $U_a^{(2)}$.

\begin{remark}\label{rem-no-perturbation}
In our setting, waves belonging to the same family do not interact. 
Consequently, unlike in the general wavefront tracking algorithm, there is no need to adjust 
wave speeds or introduce non-physical waves. 
\end{remark}

\begin{remark}
The boundary constructed in \textit{Case 3} satisfies the non-reflection boundary condition: 
waves impinging on the boundary are absorbed, and no reflected wavefronts are generated.
\end{remark}

\section{Validity of the Extension}\label{sect-extension}
To prove that the approximate piston speed $u_{\rm p}^{\mu,\Delta t}$ is well defined on $(0,\tau_{(h+1)\Delta t}]$ and 
that the approximate solutions are well defined on
\begin{equation*}
\bigcup_{k=1}^{h+1}\Omega^{\mu,\Delta t,k},
\end{equation*}
via the above
construction, it remains to establish a suitable uniform bound.

We next show that $U^{\mu,\Delta t}$ can be extended to $\Omega^{\mu,\Delta t,h+1}$ while preserving 
hypotheses $\mathrm{H}_1(h+1)$--$\mathrm{H}_3(h+1)$. 
To this end, we first introduce a Glimm-type functional and establish its non-increasing property.
This yields a uniform bound on the total variation 
of the approximate solutions. We then prove that the wavefronts
of the same family do not interact within $\Omega^{\mu,\Delta t,h+1}$. 

The following lemma is used in our proofs:
\begin{lemma}\label{lem:bv-estimate}
If $U_2 =\Phi_1(\alpha_1;\,\Xi_2(-\alpha_2;\,U_1))$ with both $U_1, U_2 \in O_{\varepsilon_0}(\overline{U})$, then
\begin{equation*}
|U_1 - U_2| \le C_1 (|\alpha_1|+|\alpha_2|),
\end{equation*}
where 
$C_1> 0$ depends only on $\overline{U}$ and the system \eqref{eq:p-system}.
\end{lemma}

\subsection{Glimm-type functionals}
We now introduce the notion of approaching waves, which is crucial in the definition of the Glimm-type functionals (also see \cite{Dafermos2026,Glimm1965}).

\begin{figure}[ht]
\centering  
\includegraphics[width=0.65\textwidth]{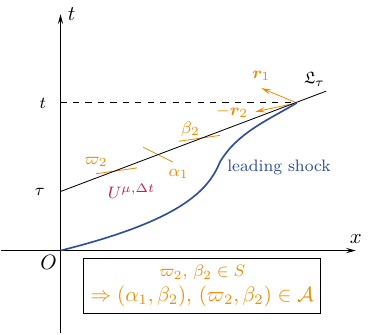}
\caption{Approaching waves}
\label{fig:approaching-wave}
\end{figure}
\begin{definition}[Approaching] Two fronts $\alpha$ and $\varpi$ located on the same line $\LL_{\tau}$ at 
positions $(x_{\alpha},t_{\alpha})\in \LL_{\tau}$ and $(x_{\varpi},t_{\varpi})\in \LL_{\tau}$ of families
$i_{\alpha}, i_{\varpi} \in \{1,\,2\}$, respectively, are \textit{approaching} if $i_{\alpha} > i_{\varpi}$ and $t_{\alpha} > t_{\varpi}$. 
This approaching relation is denoted by $(\alpha,\varpi)\in \mathcal{A}$. In particular, $\mathcal{A}(\alpha)$ denotes 
the set of $\varpi$ satisfying $(\alpha,\varpi)\in \mathcal{A}$ (see Fig. \ref{fig:approaching-wave}). 
\end{definition}

\begin{remark}
In the classical definition of approaching waves, the case $i_{\alpha} = i_{\varpi}$ is included whenever at least 
one of the corresponding wavefronts is a shock. 
In our setting, however, this case is excluded, since we will show that wavefronts of the same family never intersect.
\end{remark}

The Glimm-type functionals are defined as follows:
\begin{definition}[Glimm-type functionals]
\label{def:glimm-functional}
For each $\tau \in[0,\tau_{(h+1)\Delta t}]$,
define
\begin{align*}
L(\tau)={}&\sum_{\alpha\text{ is a weak wave crossing } \LL_{\tau}}|\alpha|,\\
Q(\tau)={}&\sum_{\substack{\alpha,\,\varpi\text{ are weak waves crossing } \LL_{\tau}\\
	(\alpha,\varpi)\in \mathcal{A}}}|\alpha||\varpi|,\\
	P(\tau)={}&\sum_{k\ge\min\{l\,:\,\tau<\tau_{l\Delta t}\}}|s^{\mu,\Delta t,k+1}-s^{\mu,\Delta t,k}|,\\
	F(\tau)={}&L(\tau)+K_1Q(\tau)+K_2P(\tau),
\end{align*}
where $K_1,\,K_2>0$  are to be determined later.
\end{definition}

To trace the evolution of the Glimm-type functionals, we need to analyze the wave interactions and the wave generation 
from the leading shock, as well as a weak wave impinging onto the piston boundary $\Gamma$. 
\begin{definition}\label{def-regular-interaction-time}
A point  \((\bar{x},\bar{t})\) is called a regular interaction point if one of the following interactions occurs{\rm :}
\begin{itemize}
\item[\rm (i)] two weak waves of the different families interact at \((\bar{x},\bar{t})\);
\item[\rm (ii)] new waves are generated from the leading shock at \((\bar{x},\bar{t})\);
\item[\rm (iii)] a weak wave interacts with the boundary $\{(0,t)\,:\,t\geq 0\}$ at \((\bar{x},\bar{t})\). 
\end{itemize}
Moreover, if two weak waves of the same family interact at \((\bar{x},\bar{t})\), 
then \((\bar{x},\bar{t})\) is called an irregular interaction point.
\end{definition}

Then, for the Glimm-type functional $F(\tau)$, we have
\begin{proposition}
\label{prop:glimm-functional-decreasing}
There exist positive constants $K_1$, $K_2$ suitably large and positive constant $\varepsilon_{\rm g}$ sufficiently small 
such that, for any regular interaction point $(\bar{x},\bar{t})\in \LL_{\tau}$, whenever $F(\tau-) \le \varepsilon_{\rm g}$, 
\begin{equation*}
F(\tau+) \leq F(\tau-).
\end{equation*}
\end{proposition}
\begin{proof}
Let $C > 0$ denote a universal constant depending only on $\overline{U}$ and 
system \eqref{eq:p-system}, varying when necessary.

\medskip\noindent
\textit{Case 1. Interior interactions between weak waves}: 
A weak wavefront $\alpha_i$ ($i$-family) interacts with $\varpi_j$ ($j$-family, \(j\neq i\)) at $(\bar x,\bar t)\in \LL_{\tau}$ 
for $i,\,j \in \{1,\,2\}$ and $\tau_{h\Delta t}<\tau<\tau_{(h+1)\Delta t}$. By Lemma \ref{lem:wave-interaction-estimate-1}, we have
\begin{align*}
L(\tau+) - L(\tau-) \le{}& C|\alpha_i||\varpi_j|, \\
Q(\tau+) - Q(\tau-) \le{}& -|\alpha_i||\varpi_j|+CL(\tau-)|\alpha_i||\varpi_j|,
\end{align*}
which implies
\begin{equation*}
F(\tau+) - F(\tau-) \le -\big(K_1-C-CL(\tau-)\big) |\alpha_i||\varpi_j|.
\end{equation*}

\medskip\noindent
\textit{Case 2. Weak wavefronts generated from the shock}: 
Wavefronts $\alpha_1$ and $\alpha_2$ are generated at $(\chi^{\mu,\Delta t}(h\Delta t),h\Delta t)$. A direct calculation leads to
\begin{align*}
L(\tau+) - L(\tau-) ={}& |\alpha_1|+|\alpha_2|, \\
Q(\tau+) - Q(\tau-) \le{}& (|\alpha_1|+|\alpha_2|)L(\tau-),\\
P(\tau+) - P(\tau-) \le{}& -|s^{\mu,\Delta t,h+1}-s^{\mu,\Delta t,h}|,
\end{align*}
which implies
\begin{equation*}
F(\tau+) - F(\tau-) \le (|\alpha_1|+|\alpha_2|)\big(1+K_1L(\tau-)\big)-K_2|s^{\mu,\Delta t,h+1}-s^{\mu,\Delta t,h}|.
\end{equation*}
By Lemmas \ref{lem:inverse-riemann-problem} and \ref{lem:strong-wave-interaction-estimate}, we obtain 
\begin{equation*}
F(\tau+) - F(\tau-) \le -|s^{\mu,\Delta t,h+1}-s^{\mu,\Delta t,h}|\big(K_2-C-CK_1L(\tau-)\big).
\end{equation*}

\medskip\noindent
\textit{Case 3. Weak wavefronts interact with the boundary}:
A wavefront $\alpha$ interacts with the boundary $\{(0,t)\,:\,t\geq 0\}$. In this case, we have
\begin{align*}
L(\tau+) - L(\tau-) ={}& -|\alpha|, \\
Q(\tau+) - Q(\tau-) \le{}& 0,
\end{align*}
leading to 
\begin{equation*}
F(\tau+) - F(\tau-) \le -|\alpha|.
\end{equation*}
\medskip\noindent
Therefore, by taking 
\begin{equation*}
K_1>3C,\quad\,\, K_2>3C,\quad\,\, \varepsilon_{\rm g}<\min\{1,\,\frac{1}{2K_1}\},
\end{equation*}
we complete the proof.
\end{proof}

Notice that the number of wavefronts remains unchanged after
a regular interaction between wave fronts of different families. Proposition \ref{prop:glimm-functional-decreasing} 
immediately implies the following corollary.
\begin{corollary}\label{coro:estimate-glimm-functional}
Fix $\tau_{\rm i}>0$ and assume that $\LL_{\tau}$ contains no irregular interaction points for any $\tau\in(0,\tau_{\rm i}]$. 
Then the total number of wavefronts on $\LL_{\tau}$ is finite. Moreover, if $F(0) < \varepsilon_{\mathrm{g}}$, then $F(\tau+) \leq F(0)$. 
Consequently, $U^{\mu,\Delta t}|_{\LL_{\tau}}\in O_{\varepsilon_0}(\overline{U})$, provided that $\epsilon_{\rm inv}$ is chosen sufficiently small.
\end{corollary}

\subsection{Wavefronts of the same family are separated}
In this subsection, we show that there exist sufficiently small constants $\epsilon_{\rm inv}$ and $\Lambda$ such that, 
for any $T>0$ and sufficiently small $\Delta t$, the wavefronts of the same family remain separated. 
In particular, the set of irregular interaction points is empty.
Throughout this subsection and the remainder of the paper, we regard a characteristic curve as a rarefaction front of zero strength.

We prove by contradiction. Assume that there are two wavefronts of the same family that 
intersect at \((\bar{x},\bar{t})\in\LL_{\tau_{\rm i}}\) for some $\tau_{\rm i}\in(0,\infty)$. 
Then we take 
\begin{equation*}
\tau_{\rm c}=\inf\left\{\tau>0\, :\,\, 
\begin{aligned}
&\text{two wavefronts of the same family}\\
&\text{intersect on  }\LL_{\tau}\text{ with }x>0
\end{aligned}\right\},
\end{equation*}
and then there is no such intersection on $\LL_{\tau}$ for $\tau<\tau_{\rm c}$.

Denote
\begin{equation}\label{eq:lambda-range}
\begin{aligned}
\lambda_*  ={}&\inf_{U\in O_{\varepsilon_0}(\overline{U})}\{|\lambda_1(U)|,\,|\lambda_2(U)|\}-\varepsilon_0,\\
\lambda^*={}&\sup_{U\in O_{\varepsilon_0}(\overline{U})}\{|\lambda_1(U)|,\,|\lambda_2(U)|\}+\varepsilon_0.
\end{aligned}
\end{equation}
We know that, if $\varepsilon_0$ is sufficiently small, then $\lambda^*>\lambda_*>0$.

To proceed further, let $x=\mathcal{C}_1(t)$ be a generalized 1-characteristic curve such that 
\begin{equation*}
\dot{\mathcal{C}}_1(t)\in [\lambda_1(U^{\mu,\Delta t}(\mathcal{C}_1(t)+,t)),\,\lambda_1(U^{\mu,\Delta t}(\mathcal{C}_1(t)-,t))];
\end{equation*}
see also \cite[(10.48)]{Bressan2000}. 
Assume that the generalized $1$-characteristic curve $x=\mathcal{C}_1(t)$ intersects the approximate leading shock at time $t_1$. 
Then we define
\begin{align*}
L_2(\tau;\,\mathcal{C}_1,t_1)={}&\sum_{\substack{\alpha_2\text{ crosses }\LL_{\tau}\text{ and}\\
\text{strictly above }x\,=\,\mathcal{C}_1(t),\,t\ge t_1}}|\alpha_2|,\\
F_2(\tau;\,\mathcal{C}_1,t_1)={}&L_2(\tau;\,\mathcal{C}_1,t_1)+K_1Q(\tau)+K_2P(\tau).
\end{align*}
Next, we show 
\begin{lemma}\label{lem:refined-estimates}
Let $K_1$, $K_2$, and $\varepsilon_{\rm g}$ be taken as in {\rm Proposition \ref{prop:glimm-functional-decreasing}}. 
Then
\begin{equation*}
F_2(\tau+;\,\mathcal{C}_1,t_1) \leq F_2(\tau-;\,\mathcal{C}_1,t_1)\qquad\mbox{for \(\, 0<\tau<\tau_{\rm c}\)},
\end{equation*}
\end{lemma}

\begin{proof}  The proof is similar to the proof of Proposition \ref{prop:glimm-functional-decreasing}, 
except for the case that a $2$-wave $\varpi$ on $\LL_{\tau}$ crosses $x=\mathcal{C}_1(t)$. 
In such a case,
\begin{equation*}
\Delta L_2(\tau;\,\mathcal{C}_1,t_1)= -|\varpi|, \qquad K_1\Delta Q(\tau)+K_2 \Delta P(\tau)\le 0.
\end{equation*} 
Then the proof is complete.
\end{proof}

Under the assumptions of Lemma \ref{lem:refined-estimates}, it is direct to show
\begin{lemma}\label{coro:total-wave-cross-C1-bounded}
The total strength of $2$-wavefronts that cross $x=\mathcal{C}_1(t)$ for \( t\geq t_1\) is bounded{\rm :}
\begin{equation*}
\sum_{\substack{\alpha_2\in\mathcal{R}\cup\mathcal{S}\text{ crossing }\\
	x\,=\, \mathcal{C}_1(t),\,t\ge t_1}}|\alpha_2|\le \varepsilon_{\rm g}
\end{equation*}
\end{lemma}
Similarly, we have 
\begin{lemma}\label{coro:total-wave-cross-C2-bounded}
The total strength of $1$-wavefronts that cross $x=\mathcal{C}_2(t)$ for \(t\geq\bar{t}\) is bounded{\rm :}
\begin{equation*}
\sum_{\substack{\varpi_1\in\mathcal{R}\cup\mathcal{S}\text{ crossing }\\
	x\,=\,\mathcal{C}_2(t),\,t\ge\bar{t}}}|\varpi_1|\le \varepsilon_{\rm g}
\end{equation*}
\end{lemma}

Moreover, we can further show that the wavefronts of the same family are separated.
\begin{lemma}\label{lem:wave-size}
Assume that $\alpha_i(\tau)\in \mathcal{R}\cup\mathcal{S}$ which is initially generated on $\LL_{\tau_0}$. 
Define 
\begin{equation*}
L_{\alpha_i}(\tau) \coloneqq \sum_{\substack{\varpi\text{ crosses } \LL_\tau\\
	\varpi \in \mathcal{A}(\alpha_i)}} |\varpi(\tau)|.
\end{equation*}
If $\alpha_i(\tau)$ never intersects with the wavefronts of family $i$, 
then there is a sufficiently large constant $M$ such that, as $\varepsilon_{\rm g}$ is sufficiently small,
\begin{enumerate}
\item[\rm (i)] $\tau\mapsto |\alpha_i(\tau)|\exp(M(L_{\alpha_i}(\tau)+K_1 Q(\tau)+K_2 P(\tau)))$ 
is monotonically decreasing for $\tau>\tau_0$.
\item[\rm (ii)]  $\tau\mapsto |\alpha_i(\tau)|\exp(-M(L_{\alpha_i}(\tau)+K_1 Q(\tau)+K_2 P(\tau)))$ 
is monotonically increasing for $\tau>\tau_0$. 
\end{enumerate} 
\end{lemma}

\begin{proof}
We trace the front and consider the interactions on $\LL_\tau$:

\medskip\noindent
\textit{Case 1. The interaction does not involve $\alpha_i$}: The interaction estimates give
\begin{equation*}
\Delta \alpha_i(\tau) = 0, \qquad \Delta L_{\alpha_i}(\tau) + K_1 \Delta Q(\tau)\le 0,
\end{equation*}
which implies
\begin{align*}
&\Delta\left(|\alpha_i(\tau)|\exp(M(L_{\alpha_i}(\tau)+K_1 Q(\tau)+K_2 P(\tau)))\right)\\
&\,\,\,=|\alpha_i(\tau-)|\,\Delta \exp(M(L_{\alpha_i}(\tau)+K_1 Q(\tau)+K_2 P(\tau)))\\
&\,\,\,\le 0,\\[2mm]
&\Delta\left(|\alpha_i(\tau)|\exp(-M(L_{\alpha_i}(\tau)+K_1 Q(\tau)+K_2 P(\tau)))\right)\\
&\,\,\,=|\alpha_i(\tau-)|\,\Delta \exp(-M(L_{\alpha_i}(\tau)+K_1 Q(\tau)+K_2 P(\tau)))\\
&\,\,\, \ge 0.
\end{align*}

\medskip
\noindent
\textit{Case 2. $\alpha_i$ collides with a weak wave $\varpi$ of the other family}: Then
\begin{equation*} 
\Delta L_{\alpha_i}(\tau) = -|\varpi|, \quad \Delta Q(\tau)< 0, \quad \Delta|\alpha_i(\tau)| \le C|\alpha_i(\tau-)||\varpi|.
\end{equation*}
This leads to 
\begin{align*}
&\Delta\big(|\alpha_i(\tau)|\exp(M(L_{\alpha_i}(\tau)+K_1 Q(\tau)+K_2 P(\tau)))\big)\\
&\,\,\,= |\Delta (\alpha_i(\tau))|\,\exp(M(L_{\alpha_i}(\tau+)+K_1 Q(\tau+)+K_2 P(\tau+)))\\
&\quad\,\,\,\,+|\alpha_i(\tau-)|\, \Delta \exp(M(L_{\alpha_i}(\tau)+K_1 Q(\tau)+K_2 P(\tau)))\\
&\,\,\,\le C|\alpha_i(\tau-)||\varpi|\,\exp(M(L_{\alpha_i}(\tau+)+K_1 Q(\tau+)+K_2 P(\tau+)))\\
&\quad\,\,\,\, +|\alpha_i(\tau-)|\exp(M(L_{\alpha_i}(\tau+)+K_1 Q(\tau+)+K_2 P(\tau+)))\big(1-\exp(M|\varpi|)\big),\\[2mm]
&\Delta\big(|\alpha_i(\tau)|\exp(-M(L_{\alpha_i}(\tau)+K_1 Q(\tau)+K_2 P(\tau)))\big)\\
&\,\,\,=|\Delta (\alpha_i(\tau))|\,\exp(-M(L_{\alpha_i}(\tau+)+K_1 Q(\tau+)+K_2 P(\tau+)))\\
&\quad\,\,\,\, +|\alpha_i(\tau-)|\, \Delta \exp(-M(L_{\alpha_i}(\tau)+K_1 Q(\tau)+K_2 P(\tau)))\\
&\,\,\,\ge -C|\alpha_i(\tau-)||\varpi|\,\exp(-M(L_{\alpha_i}(\tau+)+K_1 Q(\tau+)+K_2 P(\tau+)))\\
&\quad\,\,\,\, +|\alpha_i(\tau-)|\exp(-M(L_{\alpha_i}(\tau+)+K_1 Q(\tau+)+K_2 P(\tau+)))\big(1-\exp(-M|\varpi|)\big).
\end{align*}
Therefore, noting that $\frac{x}{\ee}\le 1-\ee^{-x}$ for $0\le x\le 1$ and that $1-\ee^x\le -x$ for $x\ge 0$, 
as $M$ is sufficiently large and $|\varpi|$ is sufficiently small,
we have
\begin{align*}
&\Delta\big(|\alpha_i(\tau)|\exp(M(L_{\alpha_i}(\tau)+K_1 Q(\tau)+K_2 P(\tau)))\big)\\
&\,\,\, \le  \left((C-M)|\alpha_i(\tau-)||\varpi|\right)\, \exp(M(L_{\alpha_i}(\tau+)+K_1 Q(\tau+)+K_2 P(\tau+)))\\
&\,\,\,  <0,\\[2mm]
&\Delta\big(|\alpha_i(\tau)|\exp(-M(L_{\alpha_i}(\tau)+K_1 Q(\tau)+K_2 P(\tau)))\big)\\
&\,\,\,\ge
\big((-C+M/\ee)|\alpha_i(\tau-)||\varpi|\big)\, \exp(-M(L_{\alpha_i}(\tau+)+K_1 Q(\tau+)+K_2 P(\tau+)))\\
&\,\,\,>0.
\end{align*}
The proof is complete.
\end{proof}

We are now ready to verify that, for suitably small $\epsilon_{\rm inv}$ and $\Lambda$, and for any $T>0$, 
the wavefronts of the same family in $U^{\mu,\Delta t}$ remain separated, provided that $\Delta t$ is sufficiently small.
Suppose that two wavefronts $\mathcal{X}_1$ and $\mathcal{X}_2$ intersect with each other on $\LL_{\tau_{\rm c}}$. Then
the interaction may occur in one of the following cases:
\begin{itemize}
\item [\rm(i)]  at least one of $\mathcal{X}_1$ and $\mathcal{X}_2$ is a $1$-shock issuing from the leading shock;
\item [\rm(ii)] both $\mathcal{X}_1$ and $\mathcal{X}_2$ are $1$-rarefaction fronts issuing from the leading shock;
\item [\rm(iii)] at least one of $\mathcal{X}_1$ and $\mathcal{X}_2$ is a $2$-shock issuing from the leading shock;
\item [\rm(iv)] both $\mathcal{X}_1$ and $\mathcal{X}_2$ are $2$-rarefaction fronts issuing from the leading shock.
\end{itemize}

Since any backward $2$-shock front never interacts with another backward $2$-wavefront (due to the entropy condition for shock waves), 
and any two forward $1$-rarefaction fronts never interact by construction of the rarefaction fronts 
in \eqref{eq-accurate-riemann-solver}, it suffices to rule out cases {\rm(i)} and {\rm(iv)}.
Without loss of generality, we consider case {\rm(i)}. Suppose that $\mathcal{X}_1$ is a $1$-rarefaction 
(or $1$-characteristic) front with strength  $\sigma_1(t)$, $\mathcal{X}_2$ is a $1$-shock front with strength 
$\sigma_2(t)$.
Let $(x_1, t_1)$ and $(x_2, t_2)$ denote their respective points of generation on the approximate leading shock; 
see Fig.~\ref{fig:shock-wave-no-interaction}. By relabeling if necessary, we may assume that
\begin{equation*}
t_2=t_1+\Delta t.
\end{equation*} 

\begin{figure}[ht]
\centering  
\includegraphics[width=0.80\textwidth]{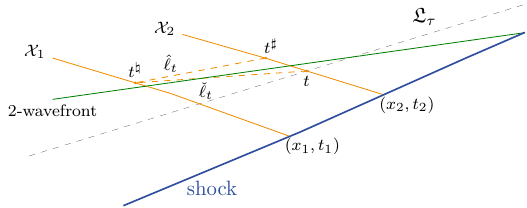}
\caption{A $1$-shock front does not intersect with another 1-wavefront}
\label{fig:shock-wave-no-interaction}
\end{figure}

We have 
\begin{proposition}\label{prop:wave-front-separate}
As $\Lambda$ and $\epsilon_{\rm inv}$ are suitably small, $\mathcal{X}_1$ and $\mathcal{X}_2$ do not intersect on \(\LL_{\tau_{\rm c}}\).
\end{proposition} 

\begin{proof} 
We know that there is no $1$-wavefront emanating from the leading shock between $\mathcal{X}_1$ and $\mathcal{X}_2$; otherwise, 
an intersection would occur on some $\LL_{\tilde{\tau}}$ with $\tilde{\tau}<\tau_{\rm c}$, contradicting the definition of $\tau_{\rm c}$.

Let $\check{\ell}_t$ be the line passing through $\mathcal{X}_2(t)$ with slope $\lambda^*$ (see \eqref{eq:lambda-range}).
This line intersects $\mathcal{X}_1$ at point $(\mathcal{X}_1(t^\n),t^\n)$.
Similarly, let $\hat{\ell}_t$ denote the line passing through $\mathcal{X}_1(t^\n)$ with slope $\lambda_*$, and
let $(\mathcal{X}_2(t^\s),t^\s)$ be its intersection point with $\mathcal{X}_2$. 

With a slight abuse of notation, we regard $\mathcal{X}_1$ and $\mathcal{X}_2$ as the corresponding maps: $\mathbb{R}_+\to \mathbb{R}$ 
that assign, to each time $t$, the $x$-coordinate of the wavefront position.
As in \cite{glimm1970,glass2007}, we define the \textit{horizontal} distance between $\mathcal{X}_1(t)$ and $\mathcal{X}_2(t)$ by
\begin{equation*}
\DD_1^*(t)=\mathcal{X}_2(t)-\mathcal{X}_1(t^\n).
\end{equation*}
Differentiating this with respect to $t$, we obtain 
\begin{equation*}
\frac{\dd}{\dd t}\DD_1^*(t)=\dot{\mathcal{X}_2}(t)-\dot{\mathcal{X}_1}(t^\n)\,\frac{\dd t^\n}{\dd t}.
\end{equation*}
Note that 
\begin{equation}\label{eq:t-flat}
\mathcal{X}_2(t)-\mathcal{X}_1(t^\n)=\lambda^*(t-t^\n),\quad
\mathcal{X}_2(t^\s)-\mathcal{X}_1(t^\n)=\lambda_*(t^\s-t^\n),
\end{equation}
which implies
\begin{equation*}
\frac{\dd t^\n}{\dd t}=\frac{\dot{\mathcal{X}_2}(t)-\lambda^*}{\dot{\mathcal{X}_1}(t^\n)-\lambda^*},\quad
\frac{\dd t^\s}{\dd t^\n}=\frac{\dot{\mathcal{X}_1}(t^\n)-\lambda_*}{\dot{\mathcal{X}_2}(t^\s)-\lambda_*}.
\end{equation*}
Therefore, we have
\begin{equation*}
\frac{\dd t^\s}{\dd t}=\frac{\dot{\mathcal{X}_2}(t)-\lambda^*}{\dot{\mathcal{X}_1}(t^\n)-\lambda^*}\, \frac{\dot{\mathcal{X}_1}(t^\n)-\lambda_*}{\dot{\mathcal{X}_2}(t^\s)-\lambda_*}>0,
\end{equation*}
so that
\begin{align}\label{eq:derivative-D}
\begin{aligned}
\frac{\dd\DD_1^*}{\dd t}={}&(\dot{\mathcal{X}_2}(t)-\dot{\mathcal{X}_1}(t^\n))\, \frac{\lambda^*}{\lambda^*-\dot{\mathcal{X}_1}(t^\n)}\\
={}&\frac{1}{2}\big(\dot{\mathcal{X}_2}(t)-\dot{\mathcal{X}_1}(t^\n)\big)\big(1+O(1)\varepsilon_0\big).
\end{aligned}  
\end{align}
As $\mathcal{X}_1$ is a $1$-rarefaction front and $\mathcal{X}_2$ is a $1$-shock front, from \eqref{eq-wave-curve} 
and the construction in \eqref{eq-accurate-riemann-solver}, we have 
\begin{align*}
\dot{\mathcal{X}_2}(t)-\dot{\mathcal{X}_1}(t^\n)
={}& \big(\dot{\mathcal{X}_2}(t)-\lambda_1(U^{\mu,\Delta t}(\mathcal{X}_2(t)-,t))\big)
-\big(\dot{\mathcal{X}_1}(t^\n)-\lambda_1(U^{\mu,\Delta t}(\mathcal{X}_1(t^\n)+,t^\n))\big)\\
&+\lambda_1(U^{\mu,\Delta t}(\mathcal{X}_2(t)-,t))-\lambda_1(U^{\mu,\Delta t}(\mathcal{X}_1(t^\n)+,t^\n))\\
={}&-\frac{1}{2}\left(\sigma_2(t)+O(1)\sigma_2^2(t)\right)+\sigma_1(t^\n)\\
&+\lambda_1(U^{\mu,\Delta t}(\mathcal{X}_2(t)-,t))-\lambda_1(U^{\mu,\Delta t}(\mathcal{X}_1(t^\n)+,t^\n)).
\end{align*}

Let $(x_\alpha, t_\alpha)$ denote the point at which wavefront $\alpha$ intersects line $\check{\ell}_t$; 
see Fig.~ \ref{fig:shock-wave-no-interaction}. 
Then 
\begin{align*}
&\lambda_1(U^{\mu,\Delta t}(\mathcal{X}_2(t)-,t))-\lambda_1(U^{\mu,\Delta t}(\mathcal{X}_1(t^\n)+,t^\n))\\
&=\sum_{ \substack{\text{wavefronts } \alpha\\\text{ crossing } \check{\ell}_t }}\lambda_1(U^{\mu,\Delta t}(x_\alpha+,t_\alpha))
-\lambda_1(U^{\mu,\Delta t}(x_\alpha-,t_\alpha))  \\
&=\sum_{ \substack{\varpi_2\in\mathcal{S}\\
	\text{crossing } \check{\ell}_t }}|\varpi_2|
	-\sum_{ \substack{\alpha_2\in\mathcal{R}\\
	\text{crossing } \check{\ell}_t }}|\alpha_2| 
	\\
	&\ge -\sum_{ \substack{\alpha_2\in\mathcal{R}\\
	\text{crossing } \check{\ell}_t }}|\alpha_2|.
\end{align*}
Consequently, we obtain 
\begin{align*}
\frac{\dd\DD_1^*}{\dd t}
\ge{}&\frac{1}{2}\bigg(-\frac{1}{2}\sigma_2(t)-\sum_{ \substack{\alpha_2\in\mathcal{R}\\
	\text{crossing } \check{\ell}_t }}|\alpha_2| \bigg)\, \big(1+O(1)\varepsilon_0\big)\\
	&+O(1) \frac{1}{2}\big(1+O(1)\varepsilon_0\big)\sigma_2^2(t) \\
	\ge{}&-\frac{1}{2}\big(1+O(1)\varepsilon_0\big)\bigg(\sigma_2(t)+\sum_{ \substack{\alpha_2\in\mathcal{R}\\
	\text{crossing } \check{\ell}_t }}|\alpha_2|\bigg),
\end{align*}
where we have used the fact that $\sigma_1$ and $\sigma_2$ are sufficiently small.

Since no $1$-shock front lies between $\mathcal{X}_1$ and $\mathcal{X}_2$, by Lemma \ref{lem:wave-size}, we have
\begin{equation*}
\sum_{ \substack{\alpha_2\in\mathcal{R}\\
	\text{crossing } \check{\ell}_t }}|\alpha_2|\leq C\sum_{ \substack{\alpha_2\in\mathcal{R}\\
	\text{crossing } \{(\mathcal{X}_2(r), r)\,:\,\,t\le r<t^\s\} }}|\alpha_2|.
\end{equation*}

Let $\alpha$ denote the $2$-rarefaction fronts that cross $\{x=\mathcal{X}_2(z),\,z\ge t_2\}$, at time $t_{\alpha}$, respectively. 
Moreover, let $t_{\alpha}^\f$ be the time such that $(t_{\alpha}^\f)^\s=t_{\alpha}$ when $\alpha\in\mathcal{R}$. Then
\begin{align}\label{eq:derivative-D-1}
\sum_{t\le t_{\alpha_2}<t^\s}|\alpha_2|
=\sum_{\substack{\alpha_2\in\mathcal{R}\\
	\text{crossing } \check{\ell}_t}} |\alpha_2|\,\mathds{1}_{(t_{\alpha_2}^\f,t_{\alpha_2}]}(t).
\end{align}

From \eqref{eq:t-flat}, we deduce 
\begin{equation*}
\frac{\lambda^*-\lambda_*}{\lambda^*+\lambda_*}\,(t-t^\n)\le t^\s-t\le \frac{\lambda^*-\lambda_*}{2\lambda_*}\,(t-t^\n).
\end{equation*}
As a result, $t_\alpha^\s-t_\alpha=O(\varepsilon_0)\DD_1^*(t_{\alpha})$. Recalling \eqref{eq:derivative-D}, we have
\begin{equation*}
\DD_1^*(z)=\DD_1^*(t_{\alpha}^\f)+O(\varepsilon_0)(z-t_{\alpha}^\f)=\big(1+O(\varepsilon_0^2)\big)\DD_1^*(t_{\alpha}^\f)
\qquad\mbox{for $z\in(t_{\alpha}^\f,t_{\alpha}]$}.
\end{equation*}

Now, \eqref{eq:derivative-D-1} becomes 
\begin{equation*}
\sum_{t\le t_{\alpha_2}<t^\s}|\alpha_2|
=\sum_{\substack{\alpha_2\in\mathcal{R}\\
	\text{crossing } \check{\ell}_t}} |\alpha_2|\,\frac{\mathds{1}_{(t_{\alpha_2}^\f,t_{\alpha_2}]}(t)}{t_{\alpha_2}-t_{\alpha_2}^\f}\, O(\varepsilon_0)\DD_1^*(t),
\end{equation*}
which gives
\begin{equation}\label{eq:no-interaction-1-shock-0}
\begin{aligned}
\frac{\dd\DD_1^*(t)}{\dd t}
\ge{}&-\sigma_2(t)\, \frac{1}{2}(1+O(1)\varepsilon_0)\\
&-C\varepsilon_0\sum_{\substack{\alpha_2\in\mathcal{R}\\
		\text{crossing } \check{\ell}_t}} |\alpha_2|\,\frac{\mathds{1}_{(t_{\alpha_2}^\f,t_{\alpha_2}]}(t)}{t_{\alpha_2}-t_{\alpha_2}^\f}\,\DD_1^*(t).
		\end{aligned}
\end{equation}

Let 
\begin{equation*}
\tilde{t}_2= t_2\, 
\frac{\lambda_*+s_0+\varepsilon_0}{\lambda_*}.
\end{equation*}
The assumption that \(\mathcal{X}_1\) and \(\mathcal{X}_2\) intersect on $(\bar{x},\bar{t})$ gives 
\begin{equation*}
\bar{x}=x_2+\int_{t_2}^{\bar{t}}\dot{\mathcal{X}}_2(z)\,\dd z.
\end{equation*}
Since $\bar{x}>0$, $x_2<t_2(s_0+\varepsilon_0)$, and $\dot{\mathcal{X}}_2(z)\le -\lambda_*$, we conclude that $\bar{t}\leq \tilde{t}_2$.

By Lemma \ref{lem:wave-size}, we see that
\begin{equation}\label{eq:1-shock-not-interact-1}
\sigma_2(t)\le C\sigma_2(t_2).
\end{equation}
Using Lemmas \ref{lem:inverse-riemann-problem}, \ref{lem:strong-wave-interaction-estimate}, and \ref{lem:approximate-leading-shock}, we have 
\begin{equation}\label{eq:1-shock-not-interact-2}
\sigma_2(t)\le \frac{C \Lambda \Delta t}{1+t_1-\Delta t}.
\end{equation}

Now, applying the Gr\"onwall's inequality to \eqref{eq:no-interaction-1-shock-0} leads to 
\begin{align*}
\DD_1^*(\bar{t})\ge{}&\exp\big(-C\varepsilon_0\,\int_{t_2}^{\bar{t}}\mathfrak{B}^\star(z)\,\dd z\big)\,\DD_1^*(t_2)\\
&- \frac{1}{2}\big(1+O(1)\varepsilon_0\big)
\int_{t_2}^{\bar{t}}\exp\big(-C\varepsilon_{0}\int_{z}^{\bar{t}}\mathfrak{B}^\star(r)\dd r\big)\sigma_2(z)\,\dd z,
\end{align*}
where 
\begin{equation*}
\mathfrak{B}^\star(z)=\sum_{\substack{\alpha_2\in\mathcal{R}\\
	\text{crossing } \check{\ell}_z}} |\alpha_2|\,\frac{\mathds{1}_{(t_{\alpha_2}^\f,t_{\alpha_2}]}(z)}{t_{\alpha_2}-t_{\alpha_2}^\f}.
\end{equation*}
According to Lemma \ref{coro:total-wave-cross-C1-bounded}, we have
\begin{equation*}
\int_{t_2}^{\bar{t}}\mathfrak{B}^\star(z)\dd z\le C\,\varepsilon_{\rm g}.
\end{equation*}
Then we conclude
\begin{align*}
\DD_1^*(\bar{t})\ge{}&\ee^{-C\varepsilon_0\varepsilon_{\rm g}}\, \DD_1^*(t_2)
-\frac{1+O(1)\varepsilon_0}{2}\,\int_{t_2}^{\bar{t}}\sigma_2(z)\,\dd z\\
\ge{}&\ee^{-C\varepsilon_0\varepsilon_{\rm g}}\, \DD_1^*(t_2)
-\frac{1+O(1)\varepsilon_0}{2}\,(\tilde{t}_2-t_2)\,\frac{C \Lambda \Delta t}{1+t_1-\Delta t}\\
\ge{}&\ee^{-C\varepsilon_0\varepsilon_{\rm g}}\, \DD_1^*(t_2)  
- C \Lambda\Delta t,
\end{align*}
where we have applied \eqref{eq:1-shock-not-interact-1}--\eqref{eq:1-shock-not-interact-2} and the fact that $t_2-t_1=\Delta t$ and $t_1\ge \Delta t$.

As $t=t_2$, observe that
\begin{equation*}
x_1+\int_{t_1}^{t_2}s^{\mu,\Delta t}(z)\dd z=x_1+\int_{t_1}^{t_2^\n}\sigma_1(z)\dd z+\lambda^*(t_2-t_2^\n).
\end{equation*}
We obtain
\begin{equation*}
(t_2-t_1)(s_0-\varepsilon_0)\le -(t_2^\n-t_1)\lambda_*+\lambda^*(t_2-t_2^\n),
\end{equation*}
which leads to
\begin{equation*}
\DD_1^*(t_2)=(t_2-t_2^\n)\lambda^*\ge \Delta t\, \frac{\lambda^*(s_0-\varepsilon_0+\lambda_*)}{\lambda^*+\lambda_*}.
\end{equation*}

At this stage, as long as we take $\epsilon_{\rm inv}$ and $\Lambda$ sufficiently small, then $\DD_1^*(\bar{t})>0$. 
The proof is complete.
\end{proof}

Next, we show {\rm (iv)}. Suppose that $\mathcal{X}_1$ and $\mathcal{X}_2$ are 2-rarefaction (2-characteristic) fronts  
of strength $\sigma_1(t)$ and $\sigma_2(t)$, generating from $(x_1,t_1)$ and $(x_2,t_2)$ on the approximate leading shock, 
respectively; see Fig.~\ref{fig:rarefaction-wave-no-interaction}. Similarly, we may assume that 
\begin{equation*}
t_2=t_1+\Delta t.
\end{equation*} 
\begin{figure}[h]
\centering  
\includegraphics[width=0.90\textwidth]{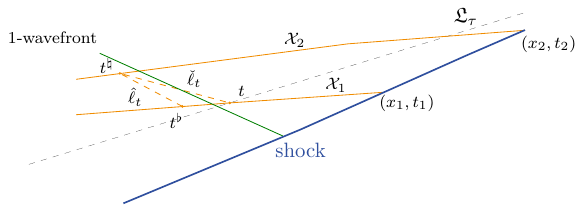}
\caption{Two 2-rarefaction fronts do not intersect}
\label{fig:rarefaction-wave-no-interaction}
\end{figure}

We have 
\begin{proposition}\label{prop:wave-front-separate-b}
As $\Lambda$ and $\epsilon_{\rm inv}$ are suitably small, $\mathcal{X}_1$ and $\mathcal{X}_2$ do not intersect on \(\LL_{\tau_{\rm c}}\).
\end{proposition} 
\begin{proof} 
We know that there is no 2-wavefront emanating from the leading shock between $\mathcal{X}_1$ and $\mathcal{X}_2$; 
otherwise, an intersection would occur on $\LL_{\tilde{\tau}}$ for some $\tilde{\tau}<\tau_{\rm c}$.

Define $\check{\ell}_t$ as the line passing through $\mathcal{X}_1(t)$ with speed $-\lambda^*$ (see \eqref{eq:lambda-range}), 
which meets $\mathcal{X}_2$ at $(\mathcal{X}_2(t^\n),t^\n)$. Similarly, let $\hat{\ell}_t$ be the line passing 
through $\mathcal{X}_2(t^\n)$ with speed $-\lambda_*$, which meets $\mathcal{X}_1$ at $(\mathcal{X}_1(t^\f),t^\f)$.

With a slight abuse of notation, we regard $\mathcal{X}_1$ and $\mathcal{X}_2$ 
as maps: $\mathbb{R}_+\to \mathbb{R}$ (mapping $t$ to the $x$-coordinate of position $\mathcal{X}_i$). 
As in \cite{glimm1970,glass2007}, define the \textit{horizontal} distance between $\mathcal{X}_1(t)$ and $\mathcal{X}_2(t)$ as 
\begin{equation*}
\DD_2^*(t)=\mathcal{X}_1(t)-\mathcal{X}_2(t^\n).
\end{equation*}
Taking the derivative with respect to $t$, we obtain 
\begin{equation*}
\frac{\dd}{\dd t}\DD_2^*(t)=\dot{\mathcal{X}}_1(t)-\dot{\mathcal{X}}_2(t^\n)\frac{\dd t^\n}{\dd t}.
\end{equation*}
Note that 
\begin{equation}\label{eq:t-flat-b}
\mathcal{X}_1(t)-\mathcal{X}_2(t^\n)=\lambda^*(t^\n-t),\quad
\mathcal{X}_1(t^\f)-\mathcal{X}_2(t^\n)=\lambda_*(t^\n-t^\f),
\end{equation}
which implies
\begin{equation*}
\frac{\dd t^\n}{\dd t}=\frac{\dot{\mathcal{X}}_1(t)+\lambda^*}{\dot{\mathcal{X}}_2(t^\n)+\lambda^*},\quad
\frac{\dd t^\f}{\dd t^\n}=\frac{\dot{\mathcal{X}}_2(t^\n)+\lambda_*}{\dot{\mathcal{X}}_1(t^\f)+\lambda_*}.
\end{equation*}
Therefore, we have
\begin{equation*}
\frac{\dd t^\f}{\dd t}=\frac{\dot{\mathcal{X}}_1(t)+\lambda^*}{\dot{\mathcal{X}}_2(t^\n)+\lambda^*}\, 
\frac{\dot{\mathcal{X}}_2(t^\n)+\lambda_*}{\dot{\mathcal{X}}_1(t^\f)+\lambda_*}>0,
\end{equation*}
and
\begin{equation}\label{eq:derivative-D-b}
\begin{aligned}
	\frac{\dd\DD_2^*}{\dd t}
	&= \bigl(\dot{\mathcal{X}}_1(t)-\dot{\mathcal{X}}_2(t^\n)\bigr)\, \frac{\lambda^*}{\dot{\mathcal{X}}_2(t^\n)+\lambda^*} \\
	&= \frac{1}{2}\bigl(\dot{\mathcal{X}}_1(t)-\dot{\mathcal{X}}_2(t^\n)\bigr)\, \bigl(1+O(1)\varepsilon_0\bigr).
\end{aligned}  
\end{equation}
As both $\mathcal{X}_1$ and $\mathcal{X}_2$ are rarefaction fronts, from \eqref{eq-wave-curve} 
and the construction in \eqref{eq:self-similar-riemann-solution}, we have 
\begin{align*}
\dot{\mathcal{X}}_1(t)-\dot{\mathcal{X}}_2(t^\n)
={}& \bigl(\dot{\mathcal{X}}_1(t)-\lambda_2(U^{\mu,\Delta t}(\mathcal{X}_1(t)-,t))\bigr) \\
&+\lambda_2(U^{\mu,\Delta t}(\mathcal{X}_1(t)-,t))-\lambda_2(U^{\mu,\Delta t}(\mathcal{X}_2(t^\n)+,t^\n))\\
={}& \sigma_1(t)+\lambda_2(U^{\mu,\Delta t}(\mathcal{X}_1(t)-,t))-\lambda_2(U^{\mu,\Delta t}(\mathcal{X}_2(t^\n)+,t^\n)).
\end{align*}

Denote by $(x_\alpha,t_\alpha)$ the position where a wavefront $\alpha$ crosses
line $\check{\ell}_t$ (see Fig.~\ref{fig:rarefaction-wave-no-interaction}). Then 
\begin{align*}
&\lambda_2(U^{\mu,\Delta t}(\mathcal{X}_1(t)-,t))-\lambda_2(U^{\mu,\Delta t}(\mathcal{X}_2(t^\n)+,t^\n))\\
={}& \sum_{\substack{\text{wavefronts } \alpha\\ \text{crossing } \check{\ell}_t}} 
\bigl(\lambda_2(U^{\mu,\Delta t}(x_\alpha+,t_\alpha))-\lambda_2(U^{\mu,\Delta t}(x_\alpha-,t_\alpha))\bigr) \\
={}& \sum_{\substack{\varpi_1\in\mathcal{S}\\ \text{crossing } \check{\ell}_t}} |\varpi_1|
-\sum_{\substack{\alpha_1\in\mathcal{R}\\ \text{crossing } \check{\ell}_t}} |\alpha_1| \\[4pt]
\leq{}& \sum_{\substack{\varpi_1\in\mathcal{S}\\ \text{crossing } \check{\ell}_t}} |\varpi_1|.
\end{align*}
Thus, we obtain 
\begin{equation*}
\frac{\dd\DD_2^*}{\dd t}
\leq\frac{1}{2}\bigg(\sigma_1(t)+\sum_{ \substack{\varpi_1\in\mathcal{S}\\
		\text{crossing } \check{\ell}_t }}|\varpi_1| \bigg) (1+O(1)\varepsilon_0).
\end{equation*}
		
		Since no 1-rarefaction front of positive strength exists between $\mathcal{X}_1$ and $\mathcal{X}_2$, 
		Lemma \ref{lem:wave-size} implies 
\begin{equation*}
\sum_{\substack{\varpi_1\in\mathcal{S}\\ \text{crossing } \check{\ell}_t}} |\varpi_1|
\leq C \sum_{\substack{\varpi_1\in\mathcal{S}\\ \text{crossing } \{(\mathcal{X}_1(r),r)\,:\,\,t^\f \le r < t\} }} |\varpi_1|.
\end{equation*}

Let $\varpi$ denote the 1-shock fronts that cross $\{x=\mathcal{X}_1(z),\,z\le t_1\}$ at time $t_{\varpi}$ (alternatively). 
Moreover, let $t_{\varpi}^\s$ be the time such that $(t_{\varpi}^\s)^\f = t_{\varpi}$ for $\varpi\in\mathcal{S}$. Then
\begin{align}\label{eq:derivative-D-1-b}
\sum_{t^\f \le t_{\varpi_1} < t} |\varpi_1|
= \sum_{\substack{\varpi_1\in\mathcal{S}\\ \text{crossing } \check{\ell}_t}} |\varpi_1| \, \mathds{1}_{[t_{\varpi_1}, t_{\varpi_1}^\s)}(t).
\end{align}

From \eqref{eq:t-flat-b}, we deduce 
\begin{equation*}
\frac{\lambda^*-\lambda_*}{\lambda^*+\lambda_*}\,(t^\n-t) \le t-t^\f \le \frac{\lambda^*-\lambda_*}{2\lambda_*}\,(t^\n-t).
\end{equation*}
Consequently, $t_\varpi - t_\varpi^\f = O(\varepsilon_0)\,\DD_2^*(t_{\varpi})$. Recalling \eqref{eq:derivative-D-b}, we obtain 
\begin{equation*}
\DD_2^*(z) = \DD_2^*(t_{\varpi}^\s) + O(\varepsilon_0)(z - t_{\varpi}^\s) = \bigl(1+O(\varepsilon_0^2)\bigr)\,\DD_2^*(t_{\varpi}^\s)
\qquad\mbox{for $z \in (t_{\varpi}, t_{\varpi}^\s]$}.
\end{equation*}

Now, \eqref{eq:derivative-D-1-b} becomes 
\begin{equation*}
\sum_{t^\f\le t_{\varpi_1}<t} |\varpi_1|
= \sum_{\substack{\varpi_1\in\mathcal{S}\\ \text{crossing } \check{\ell}_t}} 
|\varpi_1| \, \frac{\mathds{1}_{[t_{\varpi_1}, t_{\varpi_1}^\s)}(t)}{t_{\varpi_1}^\s - t_{\varpi_1}} \, O(\varepsilon_0)\,\DD_2^*(t),
\end{equation*}
which yields
\begin{equation}\label{eq:no-interaction-1-shock-0-b}
\frac{\dd\DD_2^*(t)}{\dd t}
\le C\sigma_1(t)
+ C\varepsilon_0 \sum_{\substack{\varpi_1\in\mathcal{S}\\ \text{crossing } \check{\ell}_t}} 
|\varpi_1| \, \frac{\mathds{1}_{[t_{\varpi_1}, t_{\varpi_1}^\s)}(t)}{t_{\varpi_1}^\s - t_{\varpi_1}} \, \DD_2^*(t).
\end{equation}

Let 
\begin{equation*}
\tilde{t}_1 = t_1 \, \frac{\lambda_* - s_0 - \varepsilon_0}{\lambda_*}.
\end{equation*}
The assumption that $\mathcal{X}_1$ and $\mathcal{X}_2$ intersect at $(\bar{x},\bar{t})$ gives 
\begin{equation*}
\bar{x} = x_1 + \int_{t_1}^{\bar{t}} \dot{\mathcal{X}}_1(z) \, \dd z.
\end{equation*}
Since $\bar{x} > 0$, $x_1 < t_1(s_0 + \varepsilon_0)$, and $\dot{\mathcal{X}}_1(z) \le \lambda^*$, 
then $\bar{t} \ge \tilde{t}_1$.

By Lemma \ref{lem:wave-size}, we obtain
\begin{equation}\label{eq:1-shock-not-interact-1-b}
\sigma_1(t) \le C \sigma_1(t_1).
\end{equation}
Using Lemmas \ref{lem:inverse-riemann-problem}, \ref{lem:strong-wave-interaction-estimate}, and \ref{lem:approximate-leading-shock}, we have 
\begin{equation}\label{eq:1-shock-not-interact-2-b}
\sigma_1(t)\le \frac{C \Lambda \Delta t}{1+t_1-\Delta t}.
\end{equation}
Now, applying Gr\"onwall's inequality to \eqref{eq:no-interaction-1-shock-0-b} yields 
\begin{align*}
\DD_2^*(\bar{t}) \ge{}& \exp(-C\varepsilon_0\int_{\bar{t}}^{t_1}\mathfrak{B}^\star(z)\,\dd z)\,\DD_2^*(t_1) \\
&- C \int_{\bar{t}}^{t_1}\sigma_1(z)\exp(-C\varepsilon_{0}\int_{\bar{t}}^{z}\mathfrak{B}^\star(r)\,\dd r)\,\dd z,
\end{align*}
where 
\begin{equation*}
\mathfrak{B}^\star(z)=\sum_{\substack{\varpi_1\in\mathcal{S}\\ \text{crossing } \check{\ell}_z}} 
|\varpi_1|\,\frac{\mathds{1}_{[t_{\varpi_1},t_{\varpi_1}^\s)}(z)}{t_{\varpi_1}^\s-t_{\varpi_1}}.
\end{equation*}
Note that 
\begin{equation*}
\int_{\bar{t}}^{t_2}\mathfrak{B}^\star(z)\,\dd z \le C\varepsilon_{\rm g},
\end{equation*}
by Corollary \ref{coro:total-wave-cross-C2-bounded}. We conclude
\begin{align*}
\DD_2^*(\bar{t}) &\ge \ee^{-C\varepsilon_0\varepsilon_{\rm g}} \DD_2^*(t_1)
- C\int_{\bar{t}}^{t_1}\sigma_1(z)\,\dd z \\[4pt]
&\ge \ee^{-C\varepsilon_0\varepsilon_{\rm g}} \DD_2^*(t_1)
- C(t_1-\bar{t})\,\frac{C \Lambda \Delta t}{1+t_1-\Delta t} \\[4pt]
&\ge \ee^{-C\varepsilon_0\varepsilon_{\rm g}}\DD_2^*(t_1)
- C \Lambda \Delta t,
\end{align*}
where we have applied \eqref{eq:1-shock-not-interact-1-b}--\eqref{eq:1-shock-not-interact-2-b}.

At $t=t_1$, observe that
\begin{equation*}
x_1 + \int_{t_1}^{t_2} s^{\mu,\Delta t}(z)\,\dd z 
= x_1 - \lambda^*(t_1^\n - t_1) + \int_{t_1^\n}^{t_2} \sigma_2(z)\,\dd z,
\end{equation*}
from which we obtain
\begin{equation*}
(t_2 - t_1)(s_0 + \varepsilon_0) \ge (t_2 - t_1^\n)\lambda_* - \lambda^*(t_1^\n - t_1),
\end{equation*}
which leads to
\begin{equation*}
\DD_2^*(t_1) = (t_1^\n - t_1)\lambda^* \ge \Delta t \,\frac{\lambda^*(\lambda_* - s_0 - \varepsilon_0)}{\lambda_* + \lambda^*}.
\end{equation*}

Thus, $\DD_2^*(\bar{t}) > 0$, provided that $\varepsilon_{\rm inv}$ and $\Lambda$ are chosen sufficiently small.  
This completes the proof.
\end{proof}

As a result, $U^{\mu,\Delta t}$ can be extended to $\Omega^{\mu,\Delta t,h+1}$ 
satisfying hypotheses $\mathrm{H}_1(h+1)$--$\mathrm{H}_3(h+1)$.

\section{Proof of the Main Theorem}\label{sect-proof-main-theorem}

Since no new $2$-rarefaction waves are generated through interactions between $1$-shock waves, 
and since any two $2$-rarefaction waves remain non-interacting, Lemma \ref{lem:wave-size} yields the following lemma:
\begin{lemma}\label{coro:rarefaction-front-small}
There exists a constant  $C_{\rm ra}>0$, independent of $\delta$, such that 
every rarefaction front in the approximate solution satisfies
$$
|\alpha|<C_{\rm ra}\, \delta.
$$
\end{lemma}

Moreover, we deduce a bound for the approximate piston speeds.
\begin{lemma}\label{lem:boundary-tv-estimates}
There exists $C_b>0$ depending only on $\overline{U}$ and system \eqref{eq:p-system}
such that
\begin{equation*}
TV\{u_{\rm p}^{\mu,\Delta t}:\,[0,\tau_{(h+1)\Delta t})\}<C_b F(0).
\end{equation*}
\end{lemma}
\begin{proof}
It follows from the proof of Proposition \ref{prop:glimm-functional-decreasing} that 
\begin{align*}
TV\{u_{\rm p}^{\mu,\Delta t}:\,[0,\tau_{(h+1)\Delta t})\}={}&O(1) \sum_{\alpha\text{ hits the boundary}} |\alpha|\\
\le{}&O(1)\sum_{\tau:\,\alpha \text{ hits the boundary at time }\tau} F(\tau-)-F(\tau+)\\
\le{}& C_b F(0),
\end{align*}
which completes the proof.
\end{proof}

Following the arguments in \cite{Bressan2000}, we can obtain the following proposition:
\begin{proposition}\label{prop:compactness}
Passing to suitable subsequences, we obtain the following convergences{\rm :}
\begin{enumerate}
\item[\rm (i)] the approximate solutions ${U^{\mu,\Delta t}}$ converging to a limit
\begin{equation*}
U(\cdot, \tau)\in L_{\rm loc}^1(\LL_\tau;\,\mathbb{R}^2);
\end{equation*}
\item[\rm (ii)] the approximate speed functions ${u_{\rm p}^{\mu,\Delta t}}$ converging to $u_{\rm p}\in BV(\mathbb{R}_+)$.
\end{enumerate}
\end{proposition}

Now, we are ready to prove the main theorem.
\begin{proof}[Proof of the Main Theorem]
The proof can be achieved by carrying out similar arguments as in \cite[\S 7.4]{Bressan2000}, 
together with Lemma \ref{coro:rarefaction-front-small}, \ref{lem:boundary-tv-estimates}, and Proposition \ref{prop:compactness}; 
thus we omit the details here.
\end{proof}

\appendix

\section{A lemma for interaction estimates}
The following lemma is useful in wave interaction estimates; see also \cite[Lemma 2.5]{Bressan2000}.
\begin{lemma}\label{lem:quadratic}
Suppose that $\Psi\colon\,\mathbb{R}^m\times\mathbb{R}^n\to\mathbb{R}^d$ is twice differentiable whose second derivatives are Lipschitz-continuous. Then
\begin{equation*}
\Psi(\bm x,\bm y)=\Psi(\bm 0,\bm y)+\Psi(\bm x,\bm 0)-\Psi(\bm 0,\bm 0)+O(1)|\bm x||\bm y|.
\end{equation*}
Furthermore, if 
\begin{equation*}
\frac{\partial^2\Psi}{\partial \bm x\partial\bm y }(\bm 0,\bm 0)=\bm 0,
\end{equation*}
then
\begin{equation*}
\Psi(\bm x,\bm y)=\Psi(\bm 0,\bm y)+\Psi(\bm x,\bm 0)-\Psi(\bm 0,\bm 0)+O(1)(|\bm x|+|\bm y|)|\bm x||\bm y|.
\end{equation*}
Here $O(1)$ depends only on the Lipschitz constant for the second derivatives of $\Psi$.
\end{lemma}

\bibliographystyle{plain}

\end{document}